\documentclass[11pt]{article}

\usepackage{amsfonts,amssymb, latexsym, amsmath, amsthm,mathrsfs, verbatim,geometry}
\usepackage{graphics}
\usepackage{slashed}
\usepackage{url,color}
\usepackage{enumerate}
\usepackage{esint}
\usepackage{pifont}
\usepackage{upgreek}
\usepackage{times}
\usepackage{calligra}
\usepackage{graphicx}
\usepackage{caption}
\usepackage{tikz}
\usetikzlibrary{calc}

\usepackage{hyperref}
\hypersetup{
    colorlinks,
    citecolor=red,
    filecolor=green,
    linkcolor=blue,
    urlcolor=black
}

\def\R{\mathbb R}

\newcommand\J{\mathscr J}
\def\A{\mathscr{A}}

\def\be{\begin{equation}}
\def\ee{\end{equation}}
\def\bea{\begin{eqnarray}}
\def\eea{\end{eqnarray}}
\def\beas{\begin{eqnarray*}}
\def\eeas{\end{eqnarray*}}

\def\g{\partial}

\def\l{\lambda}
\def\pa{\partial }

\def\der{X_r^a\slashed\partial^\beta}

\def\sn{\slashed\partial}

\def\norm{\mathcal S^N}
\def\vortnorm{\mathcal B^N}

\def\locnorm{S^N}
\def\locvortnorm{B^N}

\newtheorem{theorem}{Theorem}[section]

\newtheorem{proposition}[theorem]{Proposition}

\newtheorem{lemma}[theorem]{Lemma}
\newtheorem{remark}[theorem]{Remark}

\title{A class of global solutions to the Euler-Poisson system}

\author{Mahir Had\v zi\'c\thanks{Department of Mathematics, King's College London, Strand, London WC2S 2LR, UK. Email: mahir.hadzic@kcl.ac.uk.} \ and Juhi Jang\thanks{Department of Mathematics, University of Southern California, Los Angeles, CA 90089, USA, and Korea Institute for Advanced Study, Seoul, Korea.  Email: juhijang@usc.edu.}}

\date{}
\begin{document}

\maketitle

\abstract{
Using recent developments in the theory of globally defined expanding compressible gases, 
we construct a class of global-in-time solutions to the compressible 3-D Euler-Poisson system without any symmetry assumptions in both the gravitational and the plasma case.
Our allowed range of adiabatic indices includes, but is not limited to all $\gamma$ of the form $\gamma=1+\frac1n$, $n\in\mathbb N\setminus\{1\}$. The constructed solutions have initially small densities and a compact support. As $t\to\infty$ the density scatters to zero and the support grows at a linear rate in $t$.
}

\tableofcontents

\section{Introduction}

\setcounter{equation}{0}

The three dimensional compressible Euler-Poisson system couples the equation for a compressible gas to a self-consistent force field created by the gas particles: if the interaction is gravitational, we refer to the model as the {\em gravitational Euler-Poisson system} and if the interaction is electrostatic we talk about the {\em electrostatic Euler-Poisson system}.
In the gravitational case we obtain a model of a Newtonian star~\cite{ZeNo,BiTr,Ch}, while in the case of repelling forces between the particles, we arrive at a 
model for plasmas~\cite{Guo1998,GuoPausader2011}. We shall work with the free-boundary formulation of the problem, wherein a moving boundary separates the support of the gas $B(t)\subset \mathbb R^3$ from the vacuum region $\mathbb R^3\setminus B(t)$. In both cases, equations take the following form:
\begin{subequations}
\label{E:EULERPOISSON}
\begin{alignat}{2}
\g_t\rho + \text{div}\, (\rho \mathbf{u})& = 0 &&\ \text{ in } \ B(t)\,;\label{E:CONTINUITYEP}\\
\rho\left(\g_t  \mathbf{u}+ ( \mathbf{u}\cdot\nabla) \mathbf{u}\right) +\nabla p &=-\rho \nabla\Phi&&\ \text{ in } \ B(t)\,;\label{E:VELOCITYEP}\\
\Delta \Phi  = 4\pi c\,\rho, \ \lim_{|x|\to\infty}\Phi(t,x) & = 0&& \ \text{ in } \ \R^3 \,; \label{E:POISSON}\\
p&=0&& \ \text{ on } \ \partial B(t) \,;\label{E:VACUUMEP} \\
\mathcal{V}_{\partial B(t)}&= \mathbf{u}\cdot \mathbf{n}(t)  && \ \text{ on } \ \partial B(t)\,;\label{E:VELOCITYBDRYEP}\\
(\rho(0,\cdot),  \mathbf{u}(0,\cdot))=(\rho_0,  \mathbf{u}_0)\,, & \ B(0)=B_0&&\,.\label{E:INITIALEP}
\end{alignat}
\end{subequations}
Here $\rho,{\bf u}, p, \Phi$ denote the gas density, velocity, pressure, and the gravitational/electrostatic potential respectively. A further unknown is the moving domain $B(t)$ with a boundary $\partial B(t)$.  The normal velocity of $\partial B(t)$ is denoted by $\mathcal{V}_{\partial B(t)}$ and the outward pointing unit normal vector to $\partial B(t)$
by $\mathbf{n}(t)$.  The constant $c$ in the Poisson equation~\eqref{E:POISSON} is either $1$ or $-1$, corresponding to the gravitational or the plasma case respectively. 
To complete the formulation of the problem, we prescribe a polytropic equation of state:
\be\label{E:EQUATIONOFSTATE}
p = \rho^{\gamma}, \ \ 1<\gamma<\infty,
\ee
and assume that the {\em physical vacuum condition} is satisfied,
\begin{align}\label{E:PHYSICALVACUUM}
\pa_n(\rho_0^{\gamma-1})\Big|_{\partial B_0} <0.
\end{align}
System~\eqref{E:EULERPOISSON} with the polytropic equation of state~\eqref{E:EQUATIONOFSTATE}  and the physical vacuum condition~\eqref{E:PHYSICALVACUUM} will be referred to as the EP$_{\gamma}$-system.

Condition~\eqref{E:PHYSICALVACUUM} is not merely a technical, but rather a crucial requirement in the problem. It is realised for a famous class of steady states of the gravitational EP$_\gamma$ system, known as the Lane-Emden stars. For the history of the physical vacuum condition, its physical significance, and its remarkable role played in a rigorous development of the well-posedness theory for vacuum free boundary fluids, we refer the reader to~\cite{LS,L2,LY1,LY2,CLS,CoSh2011,CoSh2012,JaMa2009,JM1,JM2012,JaMa2015,LXZ,GuLe, Se, Sideris2014, HaJa,HaJa2,HaJa3}.

There are very few global existence and uniqueness results for the EP$_\gamma$-system outside of special symmetry classes; in fact, solutions could blow up in a finite time \cite{GoWe,Makino92,MaPe1990, DengXiangYang} (see~\cite{GuoZadeh1998} for the plasma case in absence of free boundaries). In this work we construct open sets of initial data that lead to global-in-time solutions in both the gravitational and the plasma case, without any symmetry assumptions.

\begin{theorem}[Main result - informal statement]\label{T:MAININFORMAL}
Let $\gamma = 1+\frac 1n$, $n\in\mathbb N\setminus\{1\}$ or $\gamma\in (1,\frac{14}{13})$, and let $c\in\{1,-1\}$. Then there exists an open set of compactly supported initial data in a suitable high-order weighted Sobolev space which lead to global-in-time solutions to the EP$_\gamma$-system. The support of these solutions expands linearly-in-time and upon a suitable rescaling it is of a nearly ellipsoidal shape.
\end{theorem}

A formal statement of this theorem is provided in Section~\ref{S:MAINRESULT}. 
A few remarks are in order

\begin{remark}
Although our analysis is carried out in Lagrangian coordinates we may infer an important consequence on the scattering behaviour of the (Eulerian) density function $t\mapsto \rho(t,\cdot)$; informally speaking
$
\rho(t,x) \sim_{t\to\infty} t^{-3}\rho_\infty(\frac{x}t)
$
for some function $\rho_\infty$ in an appropriate function class. In particular, Theorem~\ref{T:MAININFORMAL} gives a robust class of initial data that lead to solutions that scatter to zero at future infinity irrespectively of the sign of $c$. \\
\end{remark}

\begin{remark}
Our restrictions on $\gamma$ are largely technical except for the requirement $\gamma<\frac53$. See Remark~\ref{R:GAMMARESTRICTIONS}. \\
\end{remark}

\begin{remark}
It is very likely that the methods of this article are applicable to other Euler-matter models that satisfy two fundamental requirements: 1) they allow for a good well-posedness theory in the presence of vacuum free boundaries and 2) behave ``well" with respect to the scaling symmetries of the Euler flow, which is discussed at length below.
\end{remark}

Our basic new insight is that scaling symmetries and a suitable notion of criticality developed in our recent works~\cite{HaJa,HaJa2,HaJa3} allow one to identify an open set of initial data in the phase space which lead to global-in-time solutions. If interpreted correctly, for such a choice of data the gravitational/electrostatic interaction can be viewed as subcritical with respect to the gas pressure. We thus enter a regime dominated by the compressible Euler flow: 
\begin{subequations}
\label{E:EULER}
\begin{alignat}{2}
\g_t\rho + \text{div}\, (\rho \mathbf{u})& = 0 &&\ \text{ in } \ B(t)\,;\label{E:CONTINUITYE}\\
\rho\left(\pa_t  \mathbf{u}+ ( \mathbf{u}\cdot\nabla) \mathbf{u}\right) +\nabla p &= 0 &&\ \text{ in } \ B(t)\,;\label{E:VELOCITYE}\\
p&=0&& \ \text{ on } \ \partial B(t)\,; \label{E:VACUUME} \\
\mathcal{V} (\partial B(t))&= \mathbf{u}\cdot {\bf n}(t)  && \ \text{ on } \ \partial B(t)\,;\label{E:VELOCITYBDRYE}\\
(\rho(0,\cdot),  \mathbf{u}(0,\cdot))=(\rho_0,  \mathbf{u}_0)\,, & \ B(0)=B_0&&\,.\label{E:INITIALE}
\end{alignat}
\end{subequations}

If we add to it the equation of state~\eqref{E:EQUATIONOFSTATE} we refer to this system as the E$_\gamma$-system.
Recently Sideris~\cite{Sideris} discovered a family of special global-in-time solutions of the E$_\gamma$-system. This is a finite parameter family of so-called affine fluid motions whose support expands linearly-in-time and has the geometry of an ellipsoid. We have shown in~\cite{HaJa2} that the Sideris motions are nonlinearly stable in the range $1<\gamma\le\frac53$, while this statement was extended to the range $\gamma>\frac53$ in a recent work~\cite{ShSi2017}. In the absence of free boundaries global solutions were constructed in~\cite{Se1997,Grassin98,Ro}, however with unbounded velocities at spatial infinity. The affine motions from~\cite{Sideris} owe their existence to a certain quasi-conformal symmetry acting on~\eqref{E:EULER} and this symmetry is at the heart of this paper. We construct global solutions of the EP$_\gamma$-system as perturbations of the Sideris expanding solutions of the E$_\gamma$ system. This is a priori unlikely to succeed as the Sideris motions {\em do not} solve the EP$_\gamma$-system.
However, 
when $\gamma\in(1,\frac53)$ the gravitational/electrostatic field is effectively {\em subcritical} with respect to the pressure term and therefore the dynamics of Euler equation is expected to dominate 
over the force field term.

When $\gamma=\frac53$ the compressible Euler system enjoys a pseudo-conformal symmetry  first observed by Serre~\cite{Se1997}, analogous to the pseudo-conformal symmetry for the nonlinear Schr\"odinger equation.
It is a  natural threshold in our analysis which 
is best understood through the invariant scaling analysis - a detailed motivation is presented in Section~\ref{SS:MOTIVATION}.

In order to prove Theorem~\ref{T:MAININFORMAL}, we shall need a corresponding local well-posedness result. Such a theorem for compressible fluids satisfying the physical vacuum condition (and without any coupling to the gravitational/electrostatic field) was first proven by Coutand \& Shkoller~\cite{CoSh2012} and Jang \& Masmoudi~\cite{JaMa2015}. It is not surprising that both strategies in~\cite{CoSh2012,JaMa2015} can be adapted to prove a local well-posedness theorem for the Euler-Poisson system. This follows from the fact that the field term $\rho \nabla\Phi$ is of lower order with respect to the top order nonlinearity contained in the pressure term $\nabla p$, gratuity of the Poisson equation~\eqref{E:POISSON}. Nevertheless, some technical care is needed to produce the corresponding estimates for this nonlocal term. Since all the new  estimates needed for the proof of local well-posedness will be shown in the proof of Theorem~\ref{T:MAININFORMAL}, we state the local well-posedness theorem separately in the Appendix. 
 
For the spherically symmetric EP$_\gamma$-system local well-posedness is implicit in the work of Jang~\cite{Jang2014} and it was also shown by Luo, Xin, \& Zeng~\cite{LXZ}. Finally, when the underlying domain inherits the topology of the manifold $\mathbb T^2\times\mathbb R$ with a coupling to the force field given via the convolution kernel $\frac{1}{|\cdot|}$, Gu \& Lei~\cite{GuLe} showed local well-posedness relying on the framework developed in~\cite{CoSh2012}. We remark that the $\frac{1}{|\cdot|}$-kernel does not correspond to the Green function of the laplacian on $\mathbb T^2\times\mathbb R$. Nevertheless, ideas used in~\cite{GuLe} to control the nonlocal force term can be used to control the force field for the true Green's kernel, and we also use some of these ideas in our work. We additionally develop a new idea to obtain the high-order regularity in the radial (i.e. normal) direction near the boundary. To that end we use a Hodge-like decomposition argument and the intrinsic geometry generated by the background affine motion to relate the normal derivatives to the intrinsic divergence and the curl of force field, see  Lemma~\ref{L:KEY}.

Our local well-posedness theorem is stated on simply connected domains given as (sufficiently) smooth images of the unit ball in $\mathbb R^3$. This induces certain technical difficulties with respect to the existing literature - in particular we must use tangential derivatives close to the boundary and the  Cartesian derivatives away from the boundary. To accomplish this, we use cut-off functions and carefully compute nontrivial commutators that appear naturally.

In the absence of free boundaries and without any symmetry assumptions, various (typically small data) global  results for the plasma case Euler-Poisson system can be found 
in~\cite{GuoZadeh1998, Guo1998, GuoPausader2011, Jang2012, GMP2013, IonescuPausader2013, JangLiZhang2014, LiWu2014, GuoIonescuPausader2016}. For the 3D gravitational EP$_\gamma$-system, the only global  existence result available, to the best of our knowledge, is \cite{HaJa} that studied radially symmetric flows for $\gamma=\frac43$ in a vacuum free boundary framework. 

{\bf Plan of the paper.}
In Section~\ref{SS:MOTIVATION} we provide a detailed scaling analysis of the problem, explain the underlying affine motions, and state the main result in the Lagrangian coordinates.
Section~\ref{S:MAINTHM} is devoted to the nonlinear energy estimates and the proof of the main theorem. Finally, in Appendix~\ref{S:LWP} we explain how to prove a local well-posedness theorem for the free boundary Euler-Poisson system.

\section{Motivation and a precise statement of the main theorem}\label{SS:MOTIVATION}

A central problem in the theory of the free boundary Euler-Poisson system is the qualitative description of its solution space.
With a recent establishment of the well-posedness theory for data satisfying the physical vacuum condition, it is natural to ask whether there exist portions of the initial data space that lead to globally-in-time defined solutions. This paper is a contribution in this direction.

Informally speaking, the richness of possible dynamic scenarios associated with the EP$_\gamma$-system is due to a nonlinear feedback between the pressure term $\nabla p$
and the field term $\rho\nabla\Phi$ in the momentum equation~\eqref{E:VELOCITYEP}.
This is particularly well exemplified in the gravitational EP$_\gamma$ system, where the attractive gravitational force counteracts the tendency of the gas pressure to spread out the star. 
As a result, one can identify a well-known family of special steady state solutions known as the Lane-Emden stars wherein the two effects are exactly {\em balanced out} so that we obtain time-independent star solutions. 

A central question in this respect is the understanding of the phase space in the vicinity of Lane-Emden stars.  A wealth of literature from both mathematics and physics community has been devoted to the linearised stability questions for such steady states and as a result the following dichotomy emerges:
\begin{itemize}
\item When $1<\gamma<\frac43$, Lane-Emden stars are linearly unstable,
\item When $\frac43\le\gamma<2$, Lane-Emden stars are linearly stable.
\end{itemize}

In~\cite{J0,Jang2014} the second author rigorously showed that in the range $\gamma\in[\frac65,\frac43)$ the Lane-Emden stars are nonlinearly unstable, while the question of nonlinear stability 
in the range $\frac43<\gamma<2$ is open, despite the conditional stability results~\cite{Rein,LuSm}. In the critical case $\gamma=\frac43$ the associated radial Lane-Emden star is nonlinearly unstable despite the conditional linear stability. This has been essentially known since the work of Goldreich and Weber~\cite{GoWe} wherein a special class (parametrised by finitely many degrees of freedom) of both collapsing and  expanding solutions in the vicinity of the Lane-Emden stars was discovered. Nonlinear stability of the expanding stars against radially symmetric perturbations was shown by the authors~\cite{HaJa}. In an upcoming work~\cite{HaJa4} stability against general perturbations will be shown.


By contrast to the Lane-Emden stars, one may wonder whether there exist dynamic regimes, wherein the dynamics is effectively driven by the pressure term. In the absence of gravity, Sideris~\cite{Sideris} constructed a family of special globally defined affine motions, which can be realised as steady states of quasiconformally rescaled Euler system~\cite{HaJa2}.  It is thus natural to investigate the behaviour of the EP$_\gamma$-system under this rescaling. For any invertible $A\in\mathbb M^{3\times3}$ let the transformation
\be\label{E:TRANSFORMATION}
(\rho,{\bf u },\phi) \mapsto (\tilde\rho,\tilde{\bf u },\tilde\phi)
\ee
be defined  by
\begin{align}
\rho(t,x) & = \det A^{-1}\tilde\rho\left(\det A^{\frac{1-3\gamma}{6}}t,\, A^{-1}x\right) \label{E:ALMOSTRHO}\\
{\bf u}(t,x) & = \det A^{\frac{1-3\gamma}{6}}A\tilde{\bf u}\left(\det A^{\frac{1-3\gamma}{6}}t,\, A^{-1}x\right) \label{E:ALMOSTVELOCITY} \\
\Phi(t,x) & = \det A^{-\frac13} \tilde\Phi\left(\det A^{\frac{1-3\gamma}{6}}t,\, A^{-1}x\right).  \label{E:ALMOSTPOTENTIAL}
\end{align}

The resulting system for the new unknowns reads
\begin{subequations}
\label{E:GEULERPOISSON}
\begin{alignat}{2}
\g_s\tilde\rho + \text{div}\, (\tilde\rho \tilde{\mathbf{u}})& = 0 &&\ \text{ in } \ \tilde B(s)\,;\label{E:CONTINUITYG}\\
\tilde\rho\left(\pa_s  \tilde{\mathbf{u}}+ (\tilde{\mathbf{u}}\cdot\nabla) \tilde{\mathbf{u}}\right) +\Lambda\nabla (\tilde\rho^{\gamma}) 
+(\det A)^{\frac{3\gamma-4}{3}}\tilde\rho\Lambda\nabla \tilde\Phi&= 0 &&\ \text{ in } \ \tilde B(s)\,;\label{E:VELOCITYG} \\
\Lambda^{ij}\pa_{ij}\tilde\Phi &= 4\pi c \tilde\rho &&\ \text{ in } \ \tilde B(s)\,;\label{E:POTENTIALG}
\end{alignat}
\end{subequations}
where 
\begin{align}\label{E:NEWQUANITITIES}
\Lambda: = \det A^{\frac23}A^{-1}A^{-\top}, \ \ \tilde B(s) = A^{-1} B(s), \ \ s = \det A^{\frac{1-3\gamma}{6}}t.
\end{align}
We see that the term $(\det A)^{\frac{3\gamma-4}{3}}\tilde\rho\Lambda\nabla \tilde\phi$ 
is small if $\det A\gg1$ and $\gamma<\frac43$, which suggests that the gravitational/electrostatic field is negligible in this regime. 
To make this intuition explicit we seek
for a time-dependent path 
$
\mathbb R_+\ni t\mapsto A(t) \in \text{GL}^+(3)
$
such that the unknowns $(\tilde\rho,\tilde{\bf u},\tilde\phi)$ defined by  
\begin{align}
\rho(t,x) & = \det A(s)^{-1} \tilde{\rho}(s, y) , \label{E:DENSITY2}\\
\mathbf{u}(t,x) & = \det A(s)^{\frac{1-3\gamma}{6}}A(s)\tilde{\mathbf{u}}(s,y) , \label{E:VELOCITY2} \\
\Phi(t,x) & =  \det A(s)^{-\frac13} \tilde\Phi(s, y)
\end{align}
solve the Euler-Poisson system EP$_\gamma$. Here the new time and space coordinates $s$ and $y$ are given by 
\be\label{E:TIMESCALE}
\frac{ds}{dt} = \frac{1}{ \det A(t)^{\frac{3\gamma-1}{6}}}, \ \ y = A(t)^{-1}x,
\ee
motivated by the transformation~\eqref{E:ALMOSTRHO}--\eqref{E:ALMOSTVELOCITY}.
Introducing the notation
\be\label{E:MUDEF}
\mu(s) : = \det A(s)^{\frac13},
\ee 
a simple application of the chain rule transforms the equations~\eqref{E:CONTINUITYE}--\eqref{E:VELOCITYE} into
\begin{align}
&\tilde\rho_s - 3\frac{\mu_s}{\mu}\tilde\rho - A^{-1}A_sy\cdot\nabla\tilde\rho + \text{div}\,\left(\tilde\rho\tilde{\bf u}\right) =0, \label{E:CONTINUITYHAT2}\\
&\pa_s\tilde{\bf u} - \frac{3\gamma-1}{2}\frac{\mu_s}{\mu}\tilde{\bf u} + A^{-1}A_s\tilde {\bf u} + ( \tilde{\mathbf{u}}\cdot\nabla) \tilde{\mathbf{u}} 
-A^{-1}A_sy\cdot\nabla\tilde{\bf u} \notag \\
&  \ \ \ \ +\frac\gamma{\gamma-1}\Lambda\nabla(\tilde\rho^{\gamma-1}) + (\det A)^{\frac{3\gamma-4}{3}}\Lambda\nabla\tilde\Phi =0, \label{E:DENSITYHAT2}\\
&\Lambda^{ij}\pa_{ij}\tilde\Phi = 4\pi c \tilde\rho.
\end{align}
where we recall that
$
\Lambda(s) = \det A(s)^{\frac23}A(s)^{-1}A(s)^{-\top}.
$
This system of equations simplifies significantly after the introduction of the ``conformal" change of variables
\be
{\bf U}(s,y)  : = \tilde{\bf u}(s,y) +\beta(s) y,
\label{E:MODIFIEDVELOCITY}
\ee
where 
\be
 \beta(s) :=- A^{-1}A_s. \label{E:MATRIXB}
\ee
We refer to ${\bf U}$ as the modified velocity. 
Then the system~\eqref{E:CONTINUITYHAT2}--\eqref{E:DENSITYHAT2} can be rewritten as
\begin{align}
&\pa_s\tilde\rho +  \text{div}\,\left(\tilde\rho{\bf U}\right)  = 0 \label{E:CONTINUITYNEW}\\
& \pa_s{\bf U} +({\bf U}\cdot\nabla){\bf U}+\left(-\frac{3\gamma-1}{2}\frac{\mu_s}{\mu}\,\text{{\bf Id}}-2\beta(s)\right){\bf U} + \frac\gamma{\gamma-1}\Lambda\nabla(\tilde\rho^{\gamma-1}) 
+ \mu^{3\gamma-4}\Lambda\nabla\tilde\Phi\notag \\
& \ \ \ = \left[\beta_s -\left(\frac{3\gamma-1}{2}\frac{\mu_s}{\mu}\,\text{{\bf Id}}+\beta\right)\beta\right] y, \label{E:VELOCITYNEW} \\
&\Lambda^{ij}\pa_{ij}\tilde\Phi = 4\pi c \tilde\rho. \label{E:POTENTIALNEW}
\end{align}
where we have used $\text{div}\left( A^{-1}A_s  y\right) =
\frac{3\mu_s}{\mu}$ 
in verifying~\eqref{E:CONTINUITYNEW}. 
Equation~\eqref{E:POTENTIALNEW} can be solved for $\tilde\Phi$:
\begin{align}\label{E:LAMBDAPOTENTIAL}
\tilde\Phi = 4\pi c \,G_\Lambda * \tilde\rho=4\pi c \int_{\mathbb R^3}G_\Lambda(\cdot-y)\tilde\rho(y)\,dy,
\end{align}
where $G_\Lambda$ denotes the Green function associated with the operator $\Lambda^{ij}\pa_i\pa_j$.

{\bf Sideris affine motions.} If one neglects the field term $\mu^{3\gamma-4}\Lambda\nabla\tilde\Phi$ in~\eqref{E:VELOCITYNEW}, it has been shown in~\cite{HaJa2} 
that the Sideris affine motions~\cite{Sideris} can be realised as steady state solutions of the resulting system~\eqref{E:CONTINUITYNEW}--\eqref{E:VELOCITYNEW}.
Any such solution corresponds to the triple $({\bf U},\tilde \rho,\beta)$ solving
\begin{align}
{\bf U} &={\bf 0} \ \ \text{ in } B,\label{E:UEQUATION}\\
\beta_s -\left(\frac{3\gamma-1}{2}\frac{\mu_s}{\mu}\,\text{{\bf Id}}+\beta\right)\beta & =-\delta\Lambda, \ \ s\ge0, \label{E:BEQUATION}\\
 \frac\gamma{\gamma-1}\nabla(\rho^{\gamma-1}) & = - \delta y \ \ \text{ in } B, \label{E:RHOEQUATION} 
\end{align}
where we set $B=B_1(0)$ to be the unit ball in $\mathbb R^3$.
Equation~\eqref{E:RHOEQUATION} gives us an explicit enthalpy profile
$ w_\delta:=\tilde \rho^{\gamma-1}$ given by
\begin{align}\label{E:BARW}
w_\delta(y) = \delta w(y), \ \ w(y): = \frac{(\gamma-1)}{2\gamma}\left(1-| y|^2\right)_+,
\end{align}
where $f_+$ denotes the positive part of $f$.
Here the parameter $\delta$ is assumed strictly positive to ensure that the physical vacuum condition~\eqref{E:PHYSICALVACUUM} holds true.
In this work $\delta>0$ will be assumed small. As our initial density profiles will be small perturbations of $w_\delta^{\frac1{\gamma-1}}=O(\delta^{\frac1{\gamma-1}})$, this will in particular imply that our density is  small at time $t=0$. 

Converting back to the $(t,x)$-coordinates, we obtain the solutions found in~\cite{Sideris}:
\begin{align}
\rho_{A}(t,x) & = \det A(t)^{-1} \left[\frac{\delta(\gamma-1)}{2\gamma}(1-|A^{-1}(t)x|^2)\right]^{\frac{1}{\gamma-1}}, \label{E:RHOA}\\
{\bf u}_{A}(t,x) &= \dot{A}(t)A^{-1}(t)x, \label{E:UA},
\end{align}
where the matrix $A$ solves the Cauchy problem for the following second order ordinary differential equation
\begin{align}
\ddot A(t) & = \delta\det A(t)^{1-\gamma}A(t)^{-\top}, \label{E:AEQUATION}\\ 
(A(0),\dot A(0)) & = (A_0,A_1) \in \text{GL}^+(3)\times \mathbb M^{3\times3}.\label{E:AEQUATIONINITIAL}
\end{align}

\subsection{Uniform-in-$\delta$ bounds for the affine motions}\label{A:ASYMPTOTIC}

In this section we describe some of the fundamental properties of the affine motions solving~\eqref{E:AEQUATION}--\eqref{E:AEQUATIONINITIAL}. Our aim is to prove various statements about the asymptotic behaviour of the solution with constants that can be chosen uniformly-in-$\delta$.
We use the notation
\[
\|M\|^2 : = \sum_{i,j=1}^3 M_{ij}^2 = \sum_{i=1}^3 \l_i^2
\]
for the Hilbert-Schmidt norm of 
and matrix $M\in\mathbb M^{3\times3}$, where $\{\l_i\}_{i=1,2,3}$ represent the eigenvalues of $M$. It is well-known (see~\cite{Sideris,HaJa2,ShSi2017}) that initial value 
problem~\eqref{E:AEQUATION}--\eqref{E:AEQUATIONINITIAL} possesses a global solution $t\mapsto A^\delta(t)$ satisfying $\det A^\delta(t) \sim_{t\to\infty}(1+t)^3$. One may further decompose the solution in the form
$A^\delta(t) =  t b^\delta  + a^\delta(t)$, where $a^\delta,b^\delta$ are $3\times3$ matrices such that $b^\delta \in  \text{{\em GL}}^+(3)$ and moreover 
$\lim_{t\to\infty}\frac{a^\delta(t)}{1+t} =\lim_{t\to\infty}\dot a^\delta(t)=0$. For a concise proof of these statements see~\cite{Sideris} or Lemma A.1 of~\cite{HaJa2}. 
However,  the solution and therefore all the constants in the aforementioned bounds depend on the small parameter $\delta$. Our goal is to provide uniform-in-$\delta$ bounds for the $t\to\infty$-asymptotic behaviour of $A^\delta(t)$.

\begin{lemma} \label{L:ADELTAUNIFORM}
Assume that $(A_0,A_1)\in  \text{{\em GL}}^+(3)\times \text{{\em GL}}^+(3)$ are given. There exist constants $\delta^\ast, C>0$ such that for any $\delta \in(0,\delta^\ast)$
the unique solution $t\mapsto A^\delta(t)$ to the Cauchy problem
\begin{align}
A^\delta_{tt} &= \delta \left(\det{A^\delta}\right)^{1-\gamma} (A^\delta)^{-\top} \label{E:ADELTAEQ}\\
A^\delta(0) & = A_0,  \ \ A^\delta_t(0) = A_1\label{E:ADELTAINITIAL}
\end{align}
can be written in the form
\begin{align}\label{E:ABC}
A^\delta(t) = a^\delta(t) + t b^\delta, 
\ \ t\ge0,
\end{align}
where $b^\delta$ is a time-independent matrix 
and moreover
\begin{align}
\|b^\delta-A_1\| &\le C\delta, \label{E:ABUNIFORM} \\
\|\ddot{a}^\delta(t)\|& \le  C\delta (1+t)^{2-3\gamma}.\label{E:CASYMP}
\end{align}
Furthermore, the following statements hold:
\begin{enumerate}
\item[{\em (a)}] 
Let $\mu_\delta(\tau):=(\det A^\delta(\tau))^{\frac13}$ and $\mu^\delta_1: = (\det b^\delta)^{\frac13}>0$. 
Then
\begin{align}
\label{E:MU1UNIFORM}
&|\mu_1^\delta - \left(\det(A_1)\right)^{\frac13}| \le C\delta \\
&\frac1 Ce^{\mu_1^\delta\tau}\le \mu^\delta(\tau)  \le C e^{\mu_1^\delta \tau}, \ \ \tau\ge0. \label{E:MUASYMP0}
\end{align}
\item [{\em (b)}]
Furthermore, 
\begin{align}\label{E:LAMBDABOUNDSDELTA}
\|\Lambda^\delta_\tau\|+ \sum_{i=1}^3|\pa_\tau d^\delta_i| \le C e^{-\mu^\delta_1\tau} , \ \  \|\Lambda^\delta_{\tau\tau}\| \le C e^{-2\mu^\delta_0\tau}
, \ \ \|\Lambda^\delta\| + \|(\Lambda^\delta)^{-1}\| \le C,  \\
\sum_{i=1}^3\left(d^\delta_i+\frac 1{d^\delta_i}\right) \le C \label{E:EIGENVALUESBOUNDDELTA}\\
\frac1 C|{\bf w}|^2 \le \langle (\Lambda^\delta)^{-1}{\bf w}, {\bf w}\rangle  \le C |{\bf w}|^2 , \ {\bf w}\in \mathbb R^3. \label{E:LAMBDAPOSDEFDELTA}
\end{align}
where  
$d^\delta_i$, $i=1,2,3$, are the eigenvalues of the matrix $\Lambda^\delta$ and $\mu_0^\delta:=\frac{3\gamma-3}{2}\mu_1^\delta$. 
\end{enumerate}
\end{lemma}

\begin{proof}
Let $L: = A_0 + t A_1$ be the solution of the initial value problem~\eqref{E:ADELTAEQ}--\eqref{E:ADELTAINITIAL} when $\delta=0$.
Consider the Banach space of at most linearly growing continuous functions
\[
X: = \left\{\Gamma\in C\left([0,\infty),\mathbb M^{3\times 3}\right), \ \ \sup_{0\le s <\infty}\|\frac{\Gamma(s)}{1+s}\| <\infty\right\}
\]
equipped with the norm
$
\|\Gamma\|_X := \|\frac {\Gamma(t)}{1+t}\|_\infty.
$
For a constant $K>0$ to be specified below, consider a closed ball in $X$ of radius $K\delta$ i.e. let
\[
B_{K\delta} = \left\{\Gamma\in X\,\big| \sup_{0\le s} \|\frac{\Gamma(s)}{1+s}\| \le K\delta \right\}
\]
By~\eqref{E:ADELTAEQ}, $\Gamma^\delta:=A^\delta-L$ solves the ODE
\begin{align}
\ddot \Gamma^\delta(t) = \delta (1+t)^{2-3\gamma} N(\tilde L(t) + \frac{\Gamma^\delta(t)}{1+t}),
\end{align}
where $N(A) = \det A^{1-\gamma}A^{-\top}$ is the nonlinearity and $\tilde L(t):=\frac{L(t)}{1+t}$. Equivalently, one may write
\begin{align}\label{E:INTEGRALFORM}
\Gamma^\delta (t) = \delta \int_0^t \int_0^s (1+\sigma)^{2-3\gamma} N(\tilde L(\sigma) + \frac{\Gamma^\delta(\sigma)}{1+\sigma})\,d\sigma\,ds,
\end{align}
where we observe that $\Gamma^\delta$ satisfies the homogeneous boundary conditions $\Gamma^\delta(0)=0$ and $\dot \Gamma^\delta(0)=0$.
Define the operator 
\[
F(\Gamma)(t) : =  \delta \int_0^t \int_0^s (1+\sigma)^{2-3\gamma} N(\tilde L(\sigma) + \frac{\Gamma(\sigma)}{1+\sigma})\,d\sigma\,ds.
\]
Note that the matrix $\tilde L(\sigma) = \frac{L(\sigma)}{1+\sigma}$ is uniformly bounded from below and above. For any $\Gamma\in B_{K\delta}$ with $\delta$ 
sufficiently small we infer that there exist $C_1,C_2>0$ such that $C_1\le \|\tilde L(\sigma) + \frac{\Gamma(\sigma)}{1+\sigma}\|\le C_2$. By the continuity of $N$ there exists a constant
$C_3: = \max_{C_1\le \|B\|\le C_2}\|N(B)\|<\infty $ so that
\[
|F(\Gamma)(t)|  \le C_3\delta  \int_0^t \int_0^s (1+\sigma)^{2-3\gamma} \le C_4 \delta (1+t)^{4-3\gamma}.
\]
Dividing by $(1+t)$ it follows that $\sup_{0\le s\le t} \|\frac{F(\Gamma)(s)}{1+s}\| \le C_4\delta$ for some universal constant $C_4$. We may therefore choose $K = 2C_4$ and $\delta<\delta^\ast$ sufficiently small to conclude that $F$ maps $B_{K\delta}$ into itself. We next claim that $F$ is a strict contraction. To see this we note that for any $\Gamma_1,\Gamma_2\in B_{K\delta}$ we have from~\eqref{E:INTEGRALFORM}
\begin{align}
|F(\Gamma_1)(t)-F(\Gamma_2)(t)| & \le \delta  \int_0^t \int_0^s (1+\sigma)^{2-3\gamma} 
\left\vert N\left(\tilde L(\sigma)+ \frac{\Gamma_1(\sigma)}{1+\sigma}\right) - N\left(\tilde L(\sigma)+ \frac{\Gamma_2(\sigma)}{1+\sigma}\right) \right\vert \,d\sigma\,ds \notag \\
& \le \delta \max_{C_1\le \|B\|\le C_2}\|DN(B)\| \sup_{0\le s\le t} \left\vert\frac{\Gamma_1(\sigma)-\Gamma_2(\sigma)}{1+\sigma}\right\vert \int_0^t \int_0^s (1+\sigma)^{2-3\gamma} \notag \\
& \le C \delta (1+t)^{4-3\gamma} \sup_{0\le s\le t} \left\vert\frac{\Gamma_1(\sigma)-\Gamma_2(\sigma)}{1+\sigma}\right\vert. 
\end{align}
Dividing by $(1+t)$ and choosing $\delta^\ast$ sufficiently small, we conclude that the map $F$ is indeed a strict contraction. Therefore there exists a unique $\Gamma\in B_{K\delta}$ solving~\eqref{E:INTEGRALFORM} for any $\delta\le\delta^\ast$ (and $K$ depending only on $\delta^\ast$). By~\cite{Sideris} we know that there exists a 
decomposition of the form~\eqref{E:ABC} such that $b^\delta$ is time-independent and $\lim_{t\to\infty}\frac{\dot a^\delta(t)}{1+t}=0$. It follows that $\|b^\delta - A_1\|=O(\delta)$, which proves~\eqref{E:ABUNIFORM}. 
This in particular also implies~\eqref{E:MU1UNIFORM}. Estimate~\eqref{E:CASYMP}  follows easily from
$\ddot a^\delta = \delta (1+t)^{2-3\gamma}N(\tilde L(t) + \frac{\Gamma(t)}{1+t})$ and $\Gamma \in B_{K\delta}$. Bounds~\eqref{E:MUASYMP0} follow from the decomposition~\eqref{E:ABUNIFORM}, the identity $\left(\det (A^\delta)\right)^{\frac13} = (1+t)\left(\det(\frac{b^\delta t}{1+t})+ \frac{a^\delta(t)}{1+t}\right)^{\frac13}$, the bounds on $a^\delta$ afforded by~\eqref{E:CASYMP}, and the uniform bounds on $\Gamma^\delta$. Part (b) of the lemma now follows the proof of the analogous statements in Lemma A.1 of~\cite{HaJa2}. Note that in the case when the matrix $A^\delta$ is diagonal, i.e. $A^\delta(t) = \l^\delta(t)\text{Id}_{3\times3}$, then $\Lambda^\delta = \text{Id}_{3\times3}$ and part (b) is trivial. 
\end{proof} 
 
\begin{remark}
Note that the constant $C$ in the statement of Lemma~\ref{L:ADELTAUNIFORM} is independent of $\delta$.
\end{remark} 
 
We note here that
\be\label{E:MU1DELTADEF}
\mu_1^\delta = (\det b^\delta)^{\frac13}\sim (\det A_1)^{\frac13} >0
\ee
describes the leading order rate of expansion of the affine motion. For future reference
we remind the reader that 
\be\label{E:MUDELTA0DEF}
\mu_0^\delta=\frac{3\gamma-3}{2}\mu_1^\delta,
\ee
and introduce
\be\label{c2}
\mu_2^\delta:=\frac{5-3\gamma}{2}\mu_1^\delta. 
\ee 
We remark that $\mu_0^\delta$ and $\mu_2^\delta$ are both strictly positive when $1<\gamma<\frac53$.
\subsection{Lagrangian coordinates and formulation of the stability problem}\label{SS:LAGRCOORD}

We now fix a Sideris' affine motion parametrised by the choice $(A_0,A_1,\delta)\in\text{GL}^+(3)\times \mathbb M^{3\times3}\times \mathbb R_+$.
Going back to~\eqref{E:CONTINUITYNEW}--\eqref{E:POTENTIALNEW}, our strategy is 
to construct a solution to the Euler-Poisson system~\eqref{E:EULERPOISSON} as a perturbation of the prescribed Sideris motion. 
To capitalise on the background expansion of the matrix A, we shall now rephrase the problem in Lagrangian variables, following a strategy introduced in~\cite{HaJa2}.
We define the map $\eta:B \to\tilde B(s)$ as a flow map associated with the modified velocity field ${\bf U}:$
\begin{align}
\eta_s(s, y) &= \mathbf{U}(s,\eta(s, y)), \label{E:FLOWMAPV} \\
\eta(0, y) &= \eta_0( y), \label{E:ETAINITIALV}
\end{align}
where $\eta_0:B\to\tilde B(0)$ is a sufficiently smooth diffeomorphism to be specified later and 
\[
B = B_1(0)
\]
is the unit ball in $\mathbb R^3$.
To pull-back~\eqref{E:CONTINUITYNEW}-\eqref{E:VELOCITYNEW} to the fixed domain $B$, we introduce the notation
\begin{align*}
\A : = [D\eta]^{-1} \ \ &\text{ (Inverse of the Jacobian matrix)},\\
\J  : = \det [D\eta] \ \ &\text{ (Jacobian determinant)},\\
f : = \tilde\rho\circ \eta \ \ &\text{ (Lagrangian density)}, \\ 
{\bf V} : = {\bf U}\circ \eta \ \ &\text{ (Lagrangian modified velocity)},\\
\Psi^\ast : = \tilde\Phi\circ\eta \ \ &\text{ (Lagrangian potential)}.
\end{align*}
From $\A [D\eta] = \text{{\bf Id}}$, one can obtain the differentiation formula for $\A$ and $\J$: 
\[
\pa \A^k_i = - \A^k_\ell \pa \eta^\ell,_s \A^s_i \  ; \quad \pa \J = \J \A^s_\ell\pa\eta^\ell,_s 
\]
for $\pa=\pa_s$ or $\pa=\pa_i$, $i=1,2,3$. Here we have used the Einstein summation convention and the notation $F,_k$ to denote the $k^{th}$ partial  derivative of $F$. Both expressions will be used throughout the paper.

It is well-known~\cite{CoSh2012,JaMa2015} that the continuity equation~\eqref{E:CONTINUITYNEW} reduces to the relationship
\[
f \J = f_0\J_0. 
\]
We choose $\eta_0$ such that  
$ \delta w  = (f_0 \J_0)^{\gamma-1}$ where $ w$ is given in \eqref{E:BARW}. For given initial density function $\rho_0$ so that $\rho_0/\rho_A$ is smooth, where $\rho_A$ is defined in~\eqref{E:RHOA}.
The existence of such $\eta_0$ follows from a result by Dacorogna \& Moser~\cite{DM}. As a consequence the Lagrangian density can be expressed as
\be \label{E:CONTINUITYLAGRV}
f=  \delta^\alpha w^\alpha \J^{-1}, \ \ \alpha : = \frac1{\gamma-1}.
\ee
This specific choice of $\eta_0$ (gauge fixing) is important for our analysis; $\eta_0(y) - y$ 
measures the initial particle displacement with respect to the background profile.

If we set $A = (\det A)^{\frac13} O$, $O\in\text{SL}^+(3)$, then a simple calculation shows that 
$\beta = - \frac{\mu_s}{\mu}\text{Id} - O^{-1}O_s$, where we recall the definition~\eqref{E:MUDEF} of $\mu$.
Using~\eqref{E:BEQUATION} equation~\eqref{E:VELOCITYNEW} takes the following form in Lagrangian coordinates: 
\begin{align}
\label{E:VELOCITYLAGR}
& \eta_{ss} +  \frac{5-3\gamma}2\frac{\mu_s}{\mu}\eta_s+ 2O^{-1}O_s\eta_s+\delta\Lambda \eta + \frac\gamma{\gamma-1}\delta\Lambda \A^\top \nabla(f^{\gamma-1}) 
+\mu^{3\gamma-4}\Lambda \A^\top \nabla\Psi^\ast= 0  \\
&\Lambda^{ij}\A^k_i(\A^\ell_j\Psi^\ast,_\ell),_k = 4\pi c \delta^\alpha w^\alpha\J^{-1}, c=\pm1 \label{E:POTENTIALLAGR}.
\end{align}
We normalise the Lagrangian Poisson equation~\eqref{E:POTENTIALLAGR} by introducing
\begin{align}\label{E:PSIPSISTAR}
\Psi : = \delta^{-\alpha} \Psi^\ast.
\end{align}
Multiplying~\eqref{E:VELOCITYLAGR} by $\delta^{-1}w^\alpha$ and using~\eqref{E:PSIPSISTAR} we obtain,
\begin{align}
 & \delta^{-1}w^\alpha \left(\pa_{ss}\eta_i-\frac{5-3\gamma}2\frac{\mu_s}{\mu}\pa_s\eta_i + 2(O^{-1}O_s)_{ij}\pa_s\eta_j+\Lambda_{ij}\eta_j\right) \notag \\
 & \ \ \ \ + \Lambda_{ij}( w^{1+\alpha} \A^k_j\J^{-\frac1\alpha}),_k 
 +\mu^{3\gamma-4}\delta^{\alpha-1} w^\alpha\Lambda_{ij} \A^k_j \Psi,_k= 0, \ \ i=1,2,3. \label{E:ETAPDE1}
\end{align}
Since in the $s$-time variable the matrix $\beta(s)$ blows up in finite time, we introduce a new time variable $\tau$ via 
\be\label{E:TIMESCALE2}
\frac{d\tau}{ds} = \det{A}^{\frac{3\gamma-3}{6}} = \mu^{\frac{3\gamma-3}{2}} \ \text{ or equivalently } \ \frac{d\tau}{dt} = \frac1{\mu}. 
\ee
Since $\mu$ grows linearly in $t$ by results from~\cite{Sideris} 
the new time $\tau$ grows like $\log t$ as $t\to\infty$ and therefore corresponds to 
a logarithmic time-scale with respect to the original time variable $t$.
Equation~\eqref{E:VELOCITYLAGR} takes the form
\be\label{E:VELOCITYLAGR2}
\mu^{3\gamma-4}\left(\mu\eta_{\tau\tau} +  \mu_\tau\eta_\tau+ 2\mu\Gamma^*\eta_\tau \right)+\delta\Lambda \eta + \frac\gamma{\gamma-1} \delta \Lambda\A^\top \nabla(f^{\gamma-1}) 
+\mu^{3\gamma-4}\Lambda \A^\top \nabla\Psi^\ast= 0,
\ee
where
$
\Gamma^* =  O^{-1}O_\tau.
$
In coordinates,
\begin{align}
 & \delta^{-1}w^\alpha\mu^{3\gamma-4}\left(\mu\pa_{\tau\tau}\eta_i+\mu_\tau\pa_\tau\eta_i+2\mu\Gamma^*_{ij}\pa_\tau\eta_j\right)+  w^\alpha\Lambda_{i\ell}\eta_\ell \notag \\
 & \ \ \ \  +  ( w^{1+\alpha} 
\Lambda_{ij}\A^k_j\J^{-\frac1\alpha}),_k +\delta^{\alpha-1}\mu^{3\gamma-4}w^\alpha\Lambda_{ij} \A^k_j \Psi,_k= 0, \ \ i=1,2,3.\label{E:ETAPDE2}
\end{align} 
Defining the perturbation 
\begin{align} 
\uptheta( \tau, y):= \eta( \tau,y) -  y
\end{align}
equations \eqref{E:POTENTIALLAGR} and~\eqref{E:ETAPDE2} take the form:
\begin{align}
& \delta^{-1}w^\alpha\mu^{3\gamma-3} \left(\pa_{\tau\tau}\uptheta_i+\frac{\mu_\tau}{\mu}\pa_\tau\uptheta_i+2\Gamma^*_{ij}\pa_\tau\uptheta_j\right) +w^\alpha \Lambda_{i\ell}\uptheta_\ell  \notag \\
& \ \ \ \ +  \left( w^{1+\alpha} \Lambda_{ij}\left(\A^k_j\J^{-\frac1\alpha}-\delta^k_j\right)\right),_k  
+\delta^{\alpha-1}\mu^{3\gamma-4}w^\alpha\Lambda_{ij} \A^k_j \Psi,_k= 0, \ \ i=1,2,3, \label{E:THETAEQUATION} \\
& \Lambda^{ij}\A^k_i(\A^\ell_j\Psi,_\ell),_k = 4\pi c w^\alpha\J^{-1}, \ \ c=\pm1, \label{E:POTENTIALEQUATION}
\end{align}
equipped with the initial conditions
\begin{align}
& \uptheta(0,  y) = \uptheta_0(  y), \ \ \uptheta_\tau(0,  y) ={\bf V}(0, y) = {\bf V}_0(  y), \ \  y\in B=B_1({\bf 0}).  \label{E:THETAINITIAL}
\end{align}
Problem~\eqref{E:THETAEQUATION}--\eqref{E:THETAINITIAL} is the Lagrangian formulation of the stability problem around a given expanding motion $(\rho_A,{\bf u}_A,A)_{A_0,A_1,\delta}$.

\subsection{Notation}\label{SS:NOTATION}

\noindent{\bf Lie derivative of the flow map.}
For vector-fields ${\bf F}:\Omega\to\mathbb R^3$, we introduce the Lie derivatives: full gradient along the flow map $\eta$
\begin{align}
[\nabla_\eta {\bf F}]^i_j := \A^s_j {\bf F}^i,_s, \ \ i,j=1,2,3, 
\end{align}
the divergence
\begin{align}
\text{div}_\eta{\bf F} := \A^s_\ell {\bf F}^\ell,_s
\end{align}
 the anti-symmetric curl matrix 
\begin{align}\label{E:BIGCURLDEF}
\left[\text{Curl}_{\eta}{\bf F}\right]^i_j :=\A^s_j{\bf F}^i,_s- \A^s_i{\bf F}^j,_s, \ \ i,j=1,2,3,
\end{align}
and the anti-symmetric $\Lambda$-curl matrix 
\begin{align}\label{E:BIGCURLDEF}
\left[\text{Curl}_{\Lambda\A}{\bf F}\right]^i_j :=\Lambda_{jm}\A^s_m{\bf F}^i,_s- \Lambda_{im}\A^s_m{\bf F}^j,_s, i,j=1,2,3.
\end{align}
We will also use $D_{\,}{\bf F}$, $\text{div}_{\,}{\bf F}$, $\text{curl}_{\,}{\bf F}$ to denote its full gradient, its divergence, and its curl:
\[
[D_{\,}{\bf F}]^i_j = {\bf F}^i,_j; \ \ \ \text{{div}}_{\,}{\bf F} ={\bf F}^\ell,_\ell; \ \  \ \left[\text{{Curl}} _{\,}{\bf F}\right]^i_j =  F^i,_j - F^j,_i, \ \ i,j=1,2,3.
\]

\noindent{\bf Function spaces.} For any measurable function $f$, any $k\in\mathbb N$, and any non-negative function $g:\Omega\to\mathbb R_+$, such that $\int_B  w(y)^{k}\,f(y)^2\,g(y)\,dy<\infty$ we introduce the notation
\begin{align}\label{E:WEIGHTEDSPACES}
\|f\|_{k,g}^2: = \int_{B} w(y)^{k}\,f(y)^2\,g(y)\,dy.
\end{align}

\noindent{\bf Vector fields near the boundary.}  
Since our analysis in the vicinity of the boundary $\pa B$ will require a careful use of tangential and normal vector fields, we shall introduce additional notation. A tangential vector field is given by
\be\label{E:TANGENTIALVF}
\sn_{ji}:= y_j \pa_i - y_i \pa_j, \ \ i,j=1,2,3.
\ee
We denote the normal vector by $X_r$: 
\be\label{E:NORMALVF}
X_r:= r\pa_r.
\ee
Finally, it is simple to check the decomposition
\be\label{E:CARTESIANVSPOLAR}
\pa_i = \frac{y_j}{r^2} \sn_{ji} + \frac{y_i}{r^2}X_r, \ \ i=1,2,3.
\ee
and the commutator identities
\begin{align}\label{com_rel}
[\sn_{ji},X_r]=0, \quad [\sn_{ji},\sn_{ik}]= \sn_{jk}, \quad [\pa_m,X_r]= \pa_m, \quad [\pa_m,\sn_{ji}]= \delta_{mj}\pa_i - \delta_{mi}\pa_j.
\end{align}

\subsection{Main result} \label{S:MAINRESULT}

Our norms will require the usage of different vector fields in the vicinity of the boundary and away from it. To that end we introduce a cut-off function $\psi\in C^\infty(\bar B,\,[0,1])$, such that $\psi = 1$ on $\{\frac34\le|y|\le1\}$ and $\psi = 0$ on $\{0\le|y|\le\frac14\}$. 
We may now define the high-order weighted Sobolev norm that measures the size of the deviation $\uptheta$. For any $N\in\mathbb N$, let
\begin{align}
& \mathcal S^N(\uptheta,{\bf V})(\tau)  = \mathcal S^N(\tau)  \notag \\ 
&: =   \sum_{a+|\beta|\le N}\sup_{0\le\tau'\le\tau}
\Big\{\delta^{-1}\mu^{3\gamma-3}\left\|\der{\bf V}\right\|_{a+\alpha,\psi}^2 +\left\|\der{\uptheta}\right\|_{a+\alpha,\psi}^2  \notag \\
& \ \ \ \  \qquad \qquad \qquad \qquad+ \left\|\nabla_\eta\der \uptheta\right\|_{a+\alpha+1,\psi}^2
 + \left\|\text{div}_\eta\der \uptheta\right\|_{a+\alpha+1,\psi}^2 \Big\} \notag \\
&\ \ \ \ +\sum_{|\nu|\le N}\sup_{0\le\tau'\le\tau}
\Big\{\delta^{-1}\mu^{3\gamma-3}\left\|\pa^\nu{\bf V}\right\|_{\alpha,1-\psi}^2+ \left\|\pa^\nu\uptheta\right\|_{\alpha,1-\psi}^2 \notag \\
& \qquad \qquad \qquad \qquad \qquad + \left\|\nabla_\eta\pa^\nu \uptheta\right\|_{\alpha+1,1-\psi}^2
+ \left\|\text{div}_\eta\pa^\nu \uptheta\right\|_{\alpha+1,1-\psi}^2\Big\}  \label{E:SNORM} 
\end{align}

\begin{remark}
Derivatives of ${\bf V}$ are additionally weighted with a negative power of the small parameter $\delta$ - this reflects the natural balance between the velocity terms and  the pressure terms in the momentum equation~\eqref{E:THETAEQUATION}.
\end{remark}

Additionally we introduce another high-order quantity measuring the modified vorticity of ${\bf V}$ which is a priori not controlled by the norm $\mathcal S^N(\tau)$: 
\begin{align}
\vortnorm[{\bf V}](\tau) & :=\sum_{a+|\beta|\le N} \sup_{0\le\tau'\le\tau}\left\|\text{Curl}_{\Lambda\A}\der{\bf V}\right\|_{a+\alpha+1,\psi}^2  +\sum_{|\nu|\le N} \sup_{0\le\tau'\le\tau} \left\|\text{Curl}_{\Lambda\A}\pa^\nu{\bf V}\right\|_{\alpha+1,1-\psi}^2.\label{E:VORTNORMDEF}
\end{align}
We also define the quantity $\vortnorm[\uptheta]$ analogously, with $\uptheta$ instead of ${\bf V}$ in the definition~\eqref{E:VORTNORMDEF}. We are now ready to state the main result. 

\begin{theorem}\label{T:MAINTHEOREM} 
Let $(A_0,A_1)\in  \text{{\em GL}}^+(3)\times \text{{\em GL}}^+(3)$ be given, let $N\in \mathbb N$ be the smallest integer satisfying $N\geq \frac{2}{\gamma-1} +12$, and let $c\in\{-1,1\}$.     
If $\gamma\in(1,\frac53)$ is such that 
\be\label{E:WBOUND}
\sum_{a+|\beta|\le N}\|\der w^\alpha\|_{\alpha+a,1} \le C_1 <\infty,  
\ee
for some constant $C_1>0$ ($w$ is given by~\eqref{E:BARW}), then there exist $\varepsilon,\delta^\ast>0$ such that for any initial data $(\uptheta_0,{\bf V}_0)$ satisfying the assumption
\begin{align}
\norm(\theta_0,{\bf V}_0) + \vortnorm({\bf V}_0) \le \varepsilon,
\end{align}
and any affine motion $(\rho_A,{\bf u}_A,A)_{A_0,A_1,\delta}$ with $\delta\in(0,\delta^\ast)$,
the associated solution 
\[
\tau\to(\uptheta(\tau,\cdot),{\bf V}(\tau,\cdot))
\] 
of~\eqref{E:THETAEQUATION}--\eqref{E:THETAINITIAL} exist for all $\tau>0$ and is unique. Moreover, there exists a constant 
$C>0$ such that 
\be\label{E:FINALENERGYBOUND}
\norm(\tau) + e^{2\mu_0^\delta\tau} \vortnorm(\tau) \le C \varepsilon, \ \ \tau\ge0,
\ee
where $\mu_0^\delta$ is defined in~\eqref{E:MUDELTA0DEF} and therefore $|\mu_0^\delta - \frac{3\gamma-3}2(\det A_1)^{\frac13}| = O(\delta)\ll1$.
\end{theorem}

\begin{remark}[Allowed polytropic indices] \label{R:GAMMARESTRICTIONS}
For any polytropic index of the form
\begin{align}
\gamma = 1 + \frac1 n, \ \ n\in\mathbb N\setminus\{1\}
\end{align}
one can check that conditions~\eqref{E:WBOUND} and $\gamma\in(1,\frac53)$ are satisfied and our theorem applies. Similarly, if $\gamma$ is chosen such that 
$\alpha\ge 13$, i.e. $\gamma\in(1,\frac{14}{13})$ it is easy to check that the condition~\eqref{E:WBOUND} is satisfied and our theorem therefore applies. All the restrictions on on the range of allowed polytropic indices apart from $\gamma<\frac53$ are merely technical and it is likely that a more refined analysis would cover the full range $\gamma\in(1,\frac53)$.
\end{remark}

\begin{remark}
Even though our norms contain $\delta$ as a parameter, we are able to prove nonlinear energy estimates with constants that do not depend on $\delta$. To do this we rely essentially on the proof scheme inspired by our earlier work on the Euler flow~\cite{HaJa2} and the uniformity-in-$\delta$ provided by Lemma~\ref{L:ADELTAUNIFORM}.
\end{remark}

\begin{remark} The smallness of $\delta$ is necessary for our continuity argument to work in the proof of the theorem. We observe that this smallness condition implies the smallness of the initial enthalpy and initial density; see \eqref{E:BARW}. The parameter $\delta$ is also related to the total mass of the gas. Our theorem demonstrates the expansion of the gas with sufficiently small total mass in the presence of self-consistent gravitational and electrostatic forces. 
\end{remark}

\begin{remark} We remark that it is not essential to use the cutoff function $\psi$ near the boundary in designing $\norm$ and $\vortnorm$. In fact, vector fields $\sn$ and $X_r$ and commutators are well-defined throughout the domain including the origin; in particular, no coordinate singularities appear. However, they give less control than the rectangular derivatives $\pa_i$ near the origin and hence we provide the interior estimates with the cutoff function $1-\psi$ as done in \cite{HaJa2}. 
\end{remark}

\noindent
{\bf A priori assumptions.}
In the proof of Theorem~\ref{T:MAINTHEOREM} we shall assume that there exists a time interval $[0,T]$, $T>0$, such that
\begin{align}
\norm(\tau)  & <\frac13, \ \ \tau\in [0,T] \label{E:APRIORI1}\\
\|\uptheta\|_{W^{2,\infty}(B)} &<\frac13, \ \  \|\mathscr J - 1\|_{W^{1,\infty}(B)}<\frac13.   \label{E:APRIORI1.1}
\end{align}
We will show that both assumptions~\eqref{E:APRIORI1} and~\eqref{E:APRIORI1.1} can be improved on the time of existence $[0,T]$, which in conjunction with a simple continuity argument, will 
retroactively justify the a priori assumptions.

\noindent
{\bf $\delta$-dependence.}
For any given $\delta$, both the solution $\uptheta$ of~\eqref{E:THETAEQUATION}--\eqref{E:THETAINITIAL}, as well as various quantities discussed in Lemma~\ref{L:ADELTAUNIFORM} depend on $\delta$. We shall from now on drop the index $\delta$ from the notation, as there will be no confusion. At times we shall refer to Lemma~\ref{L:ADELTAUNIFORM} to clarify why the constants in our estimates are $\delta$-independent.



\section{Energy estimates and proof of the main theorem}\label{S:MAINTHM}

\setcounter{equation}{0}

The Fourier transform of the Green's function $G_\Lambda$ associated with the elliptic operator $\frac{1}{4\pi}\Lambda^{ij}\pa_{ij}$ 
is found by solving 
$
-\pi  \Lambda^{ij}\xi_i\xi_j \hat G = 1
$
or, in other words
\be\label{E:GLAMBDADEF}
\hat G_\Lambda(\tau,\xi) = - \frac{1}{\pi \Lambda^{ij}\xi_i\xi_j}
\ee
Since $\Lambda^{ij}$ is a nondegenerate positive definite symmetric matrix satisfying
$
\frac1 C \text{Id} \le \Lambda \le C \text{Id},
$
for some constant $C>0$, it can be checked by means of a simple change of variables, that for any $k\in\mathbb N\cup\{0\}$ there exists a constant $C_k$ such that 
\be\label{E:GLAMBDAHOM}
|\nabla^k G_\Lambda(y)| \lesssim C_k |y|^{-k-1}.
\ee
We note that the powers of $|y|$ appearing above scale like the derivatives of $\frac{1}{|\cdot|}$, which is the fundamental solution of the Laplacian on $\mathbb R^3$.  


\subsection{Force field estimates} \label{SS:FORCEFIELD}

To address the estimates of nonlocal field terms, we fix some notation first. Let 
\be
\mathscr G_i : = \A^k_i\Psi,_k,  \ \ i=1,2,3
\ee 
be the nonlocal force term, which we will carefully estimate in the next section. Using~\eqref{E:GLAMBDADEF} we can write $\mathscr G_i$ in the form
\begin{align}
\mathscr G_i &= \int_{B}\left(\A^k_i w^\alpha\right),_k(z) G_\Lambda(\eta(y)-\eta(z)) \,dz, \label{E:GDEF} \\
& = \int_B  \left((\A^k_i-\delta^k_i) w^\alpha\right),_k(z) G_\Lambda(\eta(y)-\eta(z)) \,dz \notag \\
& \ \ \ \ + \int_B  \left(w^\alpha\right),_i \left[G_\Lambda(\eta(y)-\eta(z))-G_\Lambda(y-z)\right] \,dz + \bar\Psi,_i,  \ \ i=1,2,3, \label{E:GDEF2}
\end{align}
where $\bar\Psi$ denotes the  gravitational potential generated by the background affine motion:
\begin{align}\label{E:PSIBARFORMULA}
\bar\Psi : = w^\alpha * G_\Lambda. 
\end{align}
Note that there exists a constant $C$ independent of $\delta$ such that for any $\delta\in[0,1]$ we have the bound
\be\label{E:UNIFORMLAMBDABOUND}
CG_{\text{Id}} \ge G_{\Lambda} \ge \frac1C G_{\text{Id}}.
\ee
We have used part (b) of Lemma~\ref{L:ADELTAUNIFORM} here.

Before we continue we refer the reader to Appendix C of~\cite{HaJa2}, where the main technical tool in our estimates - the Hardy-Sobolev embeddings - are stated in detail. We shall  use them below most typically to estimate the $L^\infty$-norm of some unknown by a weighted higher-order Sobolev norm, wherein the weight is a suitable power of $w$. 

\begin{lemma}[Tangential estimates]\label{L:TANGENTIAL}
Let $\uptheta$ be a solution of~\eqref{E:THETAEQUATION}--\eqref{E:POTENTIALEQUATION} defined on a time interval $[0,T]$. Then the following bound holds
\begin{align}\label{E:TANGENTIALMAIN}
\sum_{|\beta|\leq N} \int \psi w^{\alpha} |\sn^\beta\mathscr G(\tau,\cdot) |^2 dy \lesssim  \norm(\tau) + 1, \ \ \tau\in[0,T].
\end{align}
\end{lemma}


\begin{proof}
It is easy to see by a repeated use of the chain rule that for any multi-index $\beta$ we have
\begin{align}\label{E:GFORMULA}
\left(\sn_y+\sn_z\right)^\beta G_\Lambda(\eta(y)-\eta(z)) = \sum_{k=1}^{|\beta|}\nabla^{k} G_\Lambda(\eta(y)-\eta(z)) 
\prod_{\gamma_1+\dots+\gamma_k=\beta}\left(\sn_y^{\gamma_i}\eta(y)-\sn_z^{\gamma_i}\eta(z)\right).
\end{align}
Since $|\beta|\le N$ all but at most one index $\gamma_i$ are of size smaller than $\lfloor\frac N2\rfloor$. 
For any such index,
\be\label{E:LOWERORDER}
\|\sn_y^{\gamma_i}\eta(y)-\sn_z^{\gamma_i}\eta(z)\|_{L^\infty(B)} \lesssim \|D\eta\|_{W^{N/2,\infty}(B)}|y-z| \lesssim (1+\sqrt{\norm})|y-z|\lesssim|y-z| 
\ee
where we have used the mean value theorem, the Hardy-Sobolev embedding, and the a priori assumption $\norm\le \frac13$.
By our a priori assumptions
\be\label{E:APRIORI2}
|\eta(y)-\eta(z)|\lesssim |y-z|\lesssim|\eta(y)-\eta(z)|. 
\ee
Using~\eqref{E:GFORMULA}--\eqref{E:APRIORI2} and the homogeneity property~\eqref{E:GLAMBDAHOM} it follows that 
\begin{align}\label{E:HILFE1}
\big|\left(\sn_y+\sn_z\right)^\beta\left[G_\Lambda(\eta(y)-\eta(z))\right]\big| \lesssim \frac{1}{|y-z|^2}\sum_{|\gamma|\le\beta}|\sn_y^\gamma\eta(y)-\sn_z^\gamma\eta(z)|.
\end{align}

Similarly, with the help of~\eqref{E:GLAMBDAHOM} we can show in the same way that for any $\beta, |\beta|\le N$
\begin{align}\label{E:HILFE2}
\big|\left(\sn_y+\sn_z\right)^\beta \left[\nabla G_\Lambda(\eta(y)-\eta(z))\right]\big| \lesssim \frac{1}{|y-z|^3}\sum_{|\gamma|\le\beta}|\sn_y^\gamma\eta(y)-\sn_z^\gamma\eta(z)|.
\end{align}

We now state a general identity that will be useful in our evaluation of $\sn^\beta\mathscr G$. For any sufficiently smooth functions $q_1,q_2 :\mathbb R^3\to\mathbb R$ the following identity holds
\begin{align}\label{E:TANFORMULA}
\sn_y^\beta \int_B  q_1(z)q_2(y,z)\,dz  =\sum_{\nu\le\beta} c_\nu\int_B  \sn_z^{\beta-\nu} q_1(z) (\sn_y+\sn_z)^\nu q_2(y,z)\,dz
\end{align}
for some positive universal constants $c_\nu$. To see this, note that for any $i=1,2,3$ when $\sn_y^i$ acts on $q_2(y,z)$ we 
rewrite it in the form $\sn_{ij,y}=(\sn_{ij,y}+\sn_{ij,z})-\sn_{ij,z}$ and then integrate by parts with respect to $\sn_{ij,z}$:  
\[
 - \int_B   q_1 ( \sn_{ij} q_2 )dz= \int_B   (\sn_{ij} q_1 )  q_2 dz
\] 
Iterating this procedure we arrive at~\eqref{E:TANFORMULA}.
We stress here that related ideas in the context of a domain $\mathbb T^2\times[0,1]$ was first used by Gu and Lei~\cite{GuLe} to handle singularities formally occurring when differentiating the Green kernel. We use it here in the context of the unit ball and we therefore have to use the corresponding tangential operator $\sn$ instead. 

Let $\sn_{jl}$ be a given tangential vector field for some indices $l,j\in\{1,2,3\}, l\neq j$.  Applying $\sn_{jl,y}$ to $\mathscr G_i$, using the decomposition $\eta(y)-\eta(z) = y-z + \uptheta(y)-\uptheta(z),$ and  the formula~\eqref{E:TANFORMULA} systematically, we arrive at
\begin{align}
\sn_{jl,y}\mathscr G_i  = & \int_B  (\A^k_iw^\alpha),_k\nabla G_\Lambda(\eta(y)-\eta(z))(\sn_{jl,y}\uptheta(y)-\sn_{jl,z}\uptheta(z)) \,dz \quad (=: {\bf E_1}) \notag \\
& + \int_B (\sn_{jl,y} y - \sn_{jl,z}z)\cdot\left[(\A^k_i-\delta^k_i)w^\alpha\right],_k\nabla G_\Lambda(\eta(y)-\eta(z))\,dz \quad (=: {\bf E_2})\notag \\
& + \int_B  \sn_{jl,z}\left((\A^k_i-\delta^k_i)w^\alpha\right),_kG_\Lambda(\eta(y)-\eta(z))\,dz\quad (=: {\bf E_3}) \notag \\
& + \int_B (\sn_{jl,y} y - \sn_{jl,z}z)\left[(w^\alpha),_i\left(\nabla G_\Lambda(\eta(y)-\eta(z))-\nabla G_\Lambda(y-z)\right)\right] \,dz \quad (=: {\bf E_4})\notag \\
& + \int_B  \sn_{jl,z}(w^\alpha),_i \left[ G_\Lambda(\eta(y)-\eta(z))-G_\Lambda(y-z)\right] \,dz. \quad (=: {\bf E_5}) \notag \\
& + \sn_{jl,y}\bar\Psi,_i \label{E:NABLAGDEC}
\end{align}

Given any multi-index $\gamma$, $|\gamma|=|\beta|-1$, our goal is to estimate each error term $\sn_y^{\gamma}E_i$ individually.
The term $E_1$ is representative of the types of estimates we will be using to handle the remaining error terms. 

\medskip
\noindent
{\bf Estimates for $\sn^\gamma {\bf E_1}$.} Applying $\sn_y^\gamma$ to ${\bf E_1}$ and using the formula~\eqref{E:TANFORMULA} we arrive at the following identity
\begin{align}
\sn^\gamma {\bf E_1} & = \sum_{\nu\le\gamma} c_\nu\int_B  \sn_z^{\gamma-\nu} \pa_k(\A^k_iw^\alpha)\left(\sn_y+\sn_z\right)^\nu\left[\nabla G_\Lambda(\eta(y)-\eta(z))(\sn_{jl,y}\uptheta(y)-\sn_{jl,z}\uptheta(z))\right]\,dz \notag \\
&=\sum_{\nu\le\gamma}\sum_{\mu\le\nu} c_{\nu,\mu} \int_B  \sn_z^{\gamma-\nu} \pa_k(\A^k_iw^\alpha)\left(\sn_y+\sn_z\right)^{\nu-\mu} \nabla G_\Lambda(\eta(y)-\eta(z)) \notag \\
& \quad\quad\quad\quad\quad\qquad
\left((\sn_y^\mu\sn_{jl,y}\uptheta(y)-\sn_z^\mu\sn_{jl,z}\uptheta(z)\right)\,dz \notag
\end{align}

We distinguish three cases. First let $|\mu|=\max\{|\gamma-\nu|,|\nu-\mu|,|\mu|\}$. Then since $N-1\ge|\mu|\ge\lfloor \frac N2\rfloor$, we may estimate 
\begin{align}
&\Big|\int_B  \sn_z^{\gamma-\nu} \pa_k(\A^k_iw^\alpha)\left(\sn_y+\sn_z\right)^{\nu-\mu} \nabla G_\Lambda(\eta(y)-\eta(z)) 
\left(\sn_y^\mu\sn_{jl,y}\uptheta(y)-\sn_z^\mu\sn_{jl,z}\uptheta(z)\right)\,dz \Big| \notag \\
& \lesssim \int_B  \big|\sn_z^{\gamma-\nu} \pa_k(\A^k_iw^\alpha)\big| \frac{1}{|y-z|^2}
\big|\left((\sn_y^\mu\sn_{jl,y}\uptheta(y)-\sn_z^\mu\sn_{jl,z}\uptheta(z)\right)\big| \,dz  \notag \\
&\lesssim |\sn_y^\mu\sn_{jl,y}\uptheta(y)|\Big|\sn_z^{\gamma-\nu} \pa_k(\A^k_iw^\alpha)*\frac1{|\cdot|^2} \Big|
+\|\sn_z^{\gamma-\nu} \pa_k(\A^k_iw^\alpha)\|_{L^\infty(B)}\left\vert\sn_z^\mu\sn_{jl,z}\uptheta\right\vert * \frac1{|\cdot|^2}.\label{E:MUBIG0}
\end{align}
In the second estimate we have used~\eqref{E:HILFE1} and \eqref{E:LOWERORDER}.
Moreover, recall the Young convolution inequality: for any $f_1\in L^q$, $f_2\in L^p$ we have the bound
\be\label{E:YOUNG}
\|f_1*f_2\|_{L^r(\mathbb R^3)} \lesssim \|f_1\|_{L^q(\mathbb R^3)}\|f_2\|_{L^p(\mathbb R^3)}, \ \ \frac1p+\frac1q = 1+\frac1r.
\ee
Using~\eqref{E:YOUNG} with indices $r=\infty,p=\infty,q=1$
we obtain, 
\begin{align}
\left\vert\sn_z^{\gamma-\nu} \pa_k(\A^k_iw^\alpha)\right\vert * \frac1{|\cdot|^2} &\lesssim \|\sn_z^{\gamma-\nu} \pa_k(\A^k_iw^\alpha)\|_{L^\infty(B)} \|\frac1{|\cdot|^2}\|_{L^1(B)}\notag \\
&\lesssim  \|w^{\alpha-1}\|_{L^\infty(B)}\left(1+ \sqrt{\norm}\right) \notag \\
& \lesssim 1, \label{E:YOUNGESTIMATE}
\end{align}
where we have used 
the bound $\||\sn_z^{\gamma-\nu}\A^k_i\|_{L^\infty(B)}+\||\sn_z^{\gamma-\nu} \pa_k\A^k_i\|_{L^\infty(B)}\lesssim1$ (for any $|\gamma-\nu|\le\lfloor \frac N2\rfloor$), which can be inferred from the a priori assumptions~\eqref{E:APRIORI1}--\eqref{E:APRIORI1.1} and the Hardy-Sobolev embeddings. 

Plugging~\eqref{E:YOUNGESTIMATE} into~\eqref{E:MUBIG0}, we arrive at 
\begin{align}
&\Big|\int_B  \sn_z^{\gamma-\nu} \pa_k(\A^k_iw^\alpha)\left(\sn_y+\sn_z\right)^{\nu-\mu} \nabla G_\Lambda(\eta(y)-\eta(z)) 
\left(\sn_y^\mu\sn_{jl,y}\uptheta(y)-\sn_z^\mu\sn_{jl,z}\uptheta(z)\right)\,dz \Big| \notag \\
&\lesssim \left( |\sn_y^\mu\sn_{jl,y}\uptheta(y)| +\left\vert\sn_z^\mu\sn_{jl,z}\uptheta\right\vert * \frac1{|\cdot|^2} \right) \label{E:MUBIG}.
\end{align}

We now assume that $|\nu-\mu|=\max\{|\gamma-\nu|,|\nu-\mu|,|\mu|\}$. By~\eqref{E:HILFE2} and the a priori bounds~\eqref{E:APRIORI1}--\eqref{E:APRIORI1.1} we have the bound
$
\left\vert \left(\sn_y+\sn_z\right)^{\nu-\mu} \nabla G_\Lambda(\eta(y)-\eta(z)) \right\vert \lesssim \frac1{|y-z|^2}.
$
We therefore have
the following string of estimates,
 \begin{align}
&\Big|\int_B  \sn_z^{\gamma-\nu} \pa_k(\A^k_iw^\alpha)\left(\sn_y+\sn_z\right)^{\nu-\mu} \nabla G_\Lambda(\eta(y)-\eta(z)) 
\left(\sn_y^\mu\sn_{jl,y}\uptheta(y)-\sn_z^\mu\sn_{jl,z}\uptheta(z)\right)\,dz \Big|  \notag \\
& \lesssim 
\Big|\int_B  \left\vert\sn_z^{\gamma-\nu} \pa_k(\A^k_iw^\alpha)\right|\frac{1}{|y-z|^2}\big|\sn_y^\mu\sn_{jl,y}\uptheta(y)-\sn_z^\mu\sn_{jl,z}\uptheta(z)\big| \,dz \Big| \notag \\
&\lesssim  \left(\big|\sn_y^\mu\sn_{jl,y}\uptheta(y)\big| + \|\sn_z^\mu\sn_{jl,z}\uptheta(z)\|_{L^\infty(B)}\right)  
\left\vert\sn_z^{\gamma-\nu} \pa_k(\A^k_iw^\alpha)\right\vert * \frac1{|\cdot|^2}. \label{E:NUMINUSMUBIG0}
\end{align}
We observe that for any $\l\le\nu-\mu$ the bound $\|\sn_z^\mu\sn_{jl,z}\uptheta(z)\|_{L^\infty(B)}\lesssim \sqrt{\norm}$ holds which follows from the corresponding $L^\infty$ Hardy-Sobolev embedding theorem and our definition of the norm $\norm$ with a sufficiently high $N$.
Using~\eqref{E:YOUNGESTIMATE} in~\eqref{E:NUMINUSMUBIG0} we finally arrive at 
\begin{align}
&\Big|\int_B  \sn_z^{\gamma-\nu} \pa_k(\A^k_iw^\alpha)\left(\sn_y+\sn_z\right)^{\nu-\mu} \nabla G_\Lambda(\eta(y)-\eta(z)) 
\left(\sn_y^\mu\sn_{jl,y}\uptheta(y)-\sn_z^\mu\sn_{jl,z}\uptheta(z)\right)\,dz \Big|  \notag \\
&\lesssim  |\sn_y^\mu\sn_{jl,y}\uptheta(y)| + \sqrt{\norm}.\label{E:NUMINUSMUBIG}
\end{align}

Finally if $|\gamma-\nu|=\max\{|\gamma-\nu|,|\nu-\mu|,|\mu|\}$ we proceed similarly to obtain
\begin{align}
&\Big|\int_B  \sn_z^{\gamma-\nu} \pa_k(\A^k_iw^\alpha)\left(\sn_y+\sn_z\right)^{\nu-\mu} \nabla G_\Lambda(\eta(y)-\eta(z)) 
\left(\sn_y^\mu\sn_{jl,y}\uptheta(y)-\sn_z^\mu\sn_{jl,z}\uptheta(z)\right)\,dz \Big|   \notag \\
& \lesssim  \|D\sn^{\mu}\sn_{jl}\uptheta\|_{L^\infty(B)}\Big|\int_B  \sn_z^{\gamma-\nu} \pa_k(\A^k_iw^\alpha)\frac{1}{|y-z|}\,dz\Big| \notag \\
& \lesssim \sqrt{\norm} \left\vert \sn_z^{\gamma-\nu} \pa_k(\A^k_iw^\alpha)\right\vert * \frac{1}{|\cdot|} \notag \\
& \lesssim \sqrt{\norm}. \label{E:GAMMAMINUSNUBIG}
\end{align}
In the second line we have used~\eqref{E:HILFE1} and the mean value theorem, in the third line the bound
$ \|D\sn^{\mu}\sn_{jl}\uptheta\|_{L^\infty(B)}\lesssim \sqrt{\norm}$ (which follows from the Hardy-Sobolev $L^\infty$ embedding theorem and from $|\mu|\le \lfloor \frac N2\rfloor$),
and in the last line the Young's convolution inequality analogously to~\eqref{E:YOUNGESTIMATE}. 

Summing~\eqref{E:MUBIG}--\eqref{E:GAMMAMINUSNUBIG} we can finally estimate
\begin{align}
\int_B \psi w^{\alpha}|\sn^\gamma {\bf E_1}|^2\,dy &\lesssim  \left(\norm + \sum_{|\beta|\le N}\|\sn^\beta\uptheta\|_{0,\psi}^2 + 
\sum_{|\beta|\le N}\|\sn^\beta\uptheta*\frac1{|\cdot|^2}\|_{\alpha,\psi}^2\right) \notag\\
& \lesssim \norm, \label{E:E1BOUND}
\end{align}
where we have used the Young inequality~\eqref{E:YOUNG} again, with indices $r=\infty,p=\infty,q=1$ and the notation introduced in~\eqref{E:WEIGHTEDSPACES}.

\medskip
\noindent
{\bf Estimates for $\sn^\gamma {\bf E_2}$ and $\sn^\gamma {\bf E_3}$.} The bound
\be\label{E:E23BOUND}
\|\sqrt\psi w^{\frac\alpha2}\sn^\gamma {\bf E_i}\|_{L^2(B)}^2 \lesssim \norm, \ \ i=2,3,
\ee
follows by an analogous analysis that lead to the bound~\eqref{E:E1BOUND}. Note that $\A^k_i-\delta^k_i = - \A^k_\ell \uptheta^\ell,_i$ which therefore ensures the presence of at least one copy of $\norm$ on the right-hand side of~\eqref{E:E23BOUND}.
 
\medskip
\noindent
{\bf Estimates for $\sn^\gamma {\bf E_4}$ and $\sn^\gamma {\bf E_5}$.}  
We claim that it is sufficient to prove the estimates under the simplified assumption that $\Lambda=\text{Id}$, i.e.
\be
G_\Lambda(y) = G(y) = \frac1{|y|}.
\ee
To see this, we note that 
\[
G_\Lambda = G\circ(\sqrt D O^\top), 
\]
where $\Lambda = O D O^\top$ is the orthogonal decomposition of the matrix $\Lambda$. Then a simple change of variables justifies our claim. We thereby use~\eqref{E:UNIFORMLAMBDABOUND} and Lemma~\ref{L:ADELTAUNIFORM} to show that any new constants arising from such a change of variables can be chosen to be independent of $\delta$. 
To bound  $\sn^\gamma {\bf E_4}$ we first observe that 
\[
\nabla G_\Lambda(\eta(y)-\eta(z)) - \nabla G_\Lambda(y-z) = (y-z) \frac{|\eta(y)-\eta(z)|^3 - |y-z|^3}{|y-z|^3|\eta(y)-\eta(z)|^3} + \frac{\uptheta(y)-\uptheta(z)}{|\eta(y)-\eta(z)|^3}.
\]
Since $\eta(y)-\eta(z) = y-z+\uptheta(y)-\uptheta(z)$, by the mean value theorem we have
\begin{align*}
|\eta(y)-\eta(z)|^3 - |y-z|^3 = & 3 |\uptheta(y)-\uptheta(z)|^2\int_0^1s|(y-z)+s(\uptheta(y)-\uptheta(z))|\,ds  \\
&+ 3\left(\uptheta(y)-\uptheta(z)\right)\cdot(y-z)\int_0^1|(y-z)+s(\uptheta(y)-\uptheta(z))|\,ds.
\end{align*}
Therefore
\begin{align*}
& {\bf E_4} =  3 \int_B  (w^\alpha),_i(\sn_{jl,y} y - \sn_{jl,z}z)(y-z)\frac{|\uptheta(y)-\uptheta(z)|^2\int_0^1s|(1-s)(y-z)+s(\uptheta(y)-\uptheta(z)|\,ds }{|y-z|^3|\eta(y)-\eta(z)|^3}\,dz   \\
&  \ \ \ \ + 3\int_B (w^\alpha),_i(\sn_{jl,y} y - \sn_{jl,z}z)(y-z)\frac{\left(\uptheta(y)-\uptheta(z)\right)\cdot(y-z)\int_0^1|(1-s)(y-z)+s(\uptheta(y)-\uptheta(z))|\,ds }{|y-z|^3|\eta(y)-\eta(z)|^3} \,dz  \\
& \ \ \ \ + \int_B  (w^\alpha),_i(\sn_{jl,y} y - \sn_{jl,z}z)\frac{\uptheta(y)-\uptheta(z)}{|\eta(y)-\eta(z)|^3}dz  \\
&\; \; \; \;  = :{\bf E_{4,1}} + {\bf E_{4,2}} + {\bf E_{4,3}}.
\end{align*}
Applying the formula~\eqref{E:TANFORMULA} we obtain
\begin{align}
\sn^\gamma {\bf E_{4,1}} = &\sum_{\nu\le\gamma}\sum_{\mu\le\nu}c_{\nu,\mu}\int_B  \sn^{\gamma-\nu}_z\pa_i(w^\alpha)\left(\sn_y+\sn_z\right)^{\mu}\left((\sn_{jl,y}y - \sn_{jl,z}z)(y-z)\right) \notag \\
& \ \ \ \  \left(\sn_y+\sn_z\right)^{\nu-\mu}\left[\frac{|\uptheta(y)-\uptheta(z)|^2\int_0^1s|(1-s)(y-z)+s(\uptheta(y)-\uptheta(z)|\,ds }{|y-z|^3|\eta(y)-\eta(z)|^3}\right]\,dz
\end{align}
The term in rectangular brackets has a singularity of order $|y-z|^{-3}$ and due to the presence of $\left(\sn_y+\sn_z\right)^{\mu}\left((\sn_{jl,y}y - \sn_{jl,z}z)(y-z)\right)$ the total singularity of the argument under the integral sign is $|y-z|^{-1}$. Hereby we use the property~\eqref{E:APRIORI2} and the mean value theorem to justify this claim. This does not change with repeated application of the operator $ \left(\sn_y+\sn_z\right)^{\nu-\mu}$. Systematically applying the product rule, using the Hardy-Sobolev embeddings, estimates of the type~\eqref{E:HILFE1}, and proceeding like in the proofs of~\eqref{E:MUBIG}--\eqref{E:GAMMAMINUSNUBIG} we arrive at the bound
\begin{align}
|\sn^\gamma {\bf E_{4,i}}| \lesssim \sum_{\nu\le\gamma}|\sn^{\nu}\uptheta(y)| +   \sum_{\nu\le\gamma}\Big|\sn^{\nu}\uptheta * \frac1{|\cdot|^2}\Big|, \ \ i=1,2,3.
\end{align}
Analogous estimate holds for $\sn^\gamma {\bf E_{5}}$ and it is shown in a similar way.
Therefore, like in the proof of~\eqref{E:E1BOUND} we arrive at 
\begin{align}\label{E:E45BOUND}
\int_B \psi w^{\alpha}|\sn^\gamma {\bf E_i}|^2\,dy \lesssim \norm, \ \ i=4,5.
\end{align}

\medskip
\noindent
{\bf Estimate for last term in~\eqref{E:NABLAGDEC}.}  
Recall~\eqref{E:BARW} and~\eqref{E:PSIBARFORMULA}.
Therefore, 
\begin{align}
\int_B \psi w^{\alpha}|\sn^\gamma \sn_{jl}\bar\Psi,_i |^2\,dy \lesssim 1. \label{E:FINALSIDERIS}
\end{align}

\medskip
\noindent
{\bf Conclusion.}
Summing~\eqref{E:E1BOUND},~\eqref{E:E23BOUND},~\eqref{E:E45BOUND}, and~\eqref{E:FINALSIDERIS}, we conclude the proof of the lemma.
\end{proof}

The following lemma uses mathematical induction and div-curl decomposition of the flow map to upgrade the tangential estimates from the previous lemma to the control of normal derivatives as well.

\begin{lemma}\label{L:KEY}
Let $\uptheta$ be a solution of~\eqref{E:THETAEQUATION}--\eqref{E:POTENTIALEQUATION} defined on a time interval $[0,T]$. Then for any $\tau\in[0,T]$ 
the following bounds hold:
\begin{align}
\sum_{a+|\beta|\leq N} \int_B  \psi w^{\alpha+a} |\der\mathscr G(\tau,\cdot) |^2 dy &\lesssim  \norm(\tau) +\sum_{a+|\beta|\le N}\|\der w^\alpha\|_{\alpha+a,1}^2, \  \label{E:TANEST}\\
\sum_{|\nu|\leq N} \int_B  (1-\psi) w^{\alpha} |\pa^\nu\mathscr G(\tau,\cdot) |^2 dy &\lesssim  \norm(\tau) +1. \label{E:INTERIOREST}
\end{align}
\end{lemma}


\begin{proof}
We focus first on~\eqref{E:TANEST}. The proof proceeds by induction on the number of normal derivatives. The basis of induction corresponds to case $a=0$ holds by Lemma~\ref{L:TANGENTIAL}.
We now assume that for some $0<a<N$ the following bound holds.
\begin{align}\label{E:INDUCTIVEASS}
\sum_{A=0}^a\sum_{|\beta|\leq N-a} \int_B  \psi w^{3+A} |\pa_r^A\sn^\beta\mathscr G(\tau,\cdot) |^2 dy \lesssim  \norm(\tau) 
+\sum_{a+|\beta|\le N}\|\der w^\alpha\|_{\alpha+a,1}^2, \ \ \tau\in[0,T].
\end{align}
To complete the proof we need to show that the bound~\eqref{E:INDUCTIVEASS} holds with $a$ replaced by $a+1.$
It is clear from the definition of $\mathscr G = \A\nabla \psi$ and the elliptic equation~\eqref{E:POTENTIALEQUATION} that
\begin{align}
\text{div}_{\Lambda\A} \mathscr G  &: = \Lambda_{ij}\A^k_i\pa_k\mathscr G^j =  4\pi c w^\alpha + 4\pi c w^\alpha(\J^{-1}-1) ,\label{div}\\
[\text{Curl}_\eta \mathscr G]^i_j &=0 \label{curl}.
\end{align}
On the other hand, using the identity $\A^s_k = \delta^s_k - \A^s_\ell\uptheta^\ell,_k $, 
\begin{align}
\text{div}_{\Lambda\A}  \mathscr G &=\text{div}_\Lambda \mathscr G -\Lambda_{ij}\A^k_\ell\uptheta^\ell,_i \mathscr G^j,_k \label{E:DIV2}\\ 
[\text{Curl}_\eta \mathscr G]^i_j &= [\text{Curl} \mathscr G]^i_j  - \A^s_k\theta^k,_j \mathscr G^i,_s+ \A^s_k\theta^k,_i \mathscr G^j,_s\label{E:CURL2},
\end{align}
where $\text{div}_\Lambda \mathscr G = \Lambda_{ij}\pa_i\mathscr G^j$.
From~\eqref{div}--\eqref{E:CURL2} we obtain the identities
\begin{align}
\text{div}_\Lambda \mathscr G  &=4\pi c w^\alpha + 4\pi c w^\alpha(\J^{-1}-1)  + \Lambda_{ij}\A^k_\ell\uptheta^\ell,_i \mathscr G^j,_k, \label{E:DIV3}\\ 
[\text{Curl} \mathscr G]^i_j &=  \A^s_k\theta^k,_j \mathscr G^i,_s- \A^s_k\theta^k,_i \mathscr G^j,_.s \label{E:CURL3}
\end{align}
Observe that for any $i,k,m\in\{1,2,3\}$ we have,
\begin{align*}
y^i \text{div}_\Lambda \mathscr G & = \Lambda_{km}\sn_{ik} \mathscr G^m +\Lambda_{km} y^k\pa_i\mathscr G^m \\
& = \Lambda_{km}\sn_{ik} \mathscr G^m + \Lambda_{km} [\text{Curl}\mathscr G]^m_i + \Lambda_{km}y^k\pa_m\mathscr G^i.
\end{align*}
Using~\eqref{E:CARTESIANVSPOLAR}
we may write
\[
\Lambda_{km}y^k\pa_m\mathscr G^i = \Lambda_{km}\frac{y^ky^m}{r^2}X_r \mathscr G^i + \Lambda_{km}\frac{y^ky^\ell}{r^2}\sn_{\ell m}\mathscr G^i.
\]
Letting $L:= \Lambda_{km}\frac{y^ky^m}{r^2}>0$, $k,m=1,2,3$ and combining the previous two identities, we obtain
\begin{align}
X_r\mathscr G^i = \frac1{L} \left(y^i \text{div}_\Lambda \mathscr G+ \Lambda_{km} [\text{Curl}\mathscr G]^i_m 
- \Lambda_{km}\left(\sn_{ik} \mathscr G^m +\frac{y^ky^\ell}{r^2}\sn_{\ell m}\mathscr G^i \right)  \right). \label{E:PARTIALYG}
\end{align}
We now apply the operator $\der, |\beta|\le N-a-1$ to $X_r \mathscr G^i $ and observe that by~\eqref{com_rel} 
\begin{align}
\int_B  \psi w^{1+\alpha+a}\big|X_r^{a+1}\sn^\beta \mathscr G\big|^2\,dy 
= \int_B  \psi w^{1+\alpha +a}\big|X_r^{a}\sn^\beta X_r\mathscr G\big|^2.
\label{E:PARTIALYGESTIMATE0}
\end{align}
Identity~\eqref{E:PARTIALYG} now gives
\begin{align}
 \int_B  \psi w^{1+\alpha+a}\big|X_r^{a}\sn^\beta X_r\mathscr G\big|^2\,dy
& \lesssim  \sum_{A=0}^a\sum_{|\beta'|\le N-a-1}\int_B  \psi w^{1+\alpha+a}\left(\big|X_r^{A}\sn^{\beta'} \text{div}_\Lambda\mathscr G\big|^2+\big|X_r^{A}\sn^{\beta'} \text{Curl}\mathscr G\big|^2\right) \,dy \notag \\
 & \ \ \ \ +\sum_{A=0}^a\sum_{|\beta'|\le {N-a}}\int_B  \psi w^{1+\alpha+a}\left(\big|X_r^{A}\sn^{\beta'} \mathscr G\big|^2\right)\, dy.  \label{E:PARTIALYGESTIMATE}
\end{align}
We now make a crucial use of~\eqref{E:DIV3}--\eqref{E:CURL3}.
\begin{align}
&\int_B  \psi w^{1+\alpha +a}\left(\big|X_r^{A}\sn^{\beta'} \text{div}_\Lambda\mathscr G\big|^2+\big|X_r^{A}\sn^{\beta'} \text{Curl}\mathscr G\big|^2\right) \,dy\notag \\
&\lesssim  
 \ \ \ \  
 \sum_{A=0}^a\sum_{|\beta'|\le {N-a}}\int_B  \psi w^{1+\alpha+a}\big|X_r^{A}\sn^{\beta'} D\uptheta\big|^2\, dy + \sum_{a+|\beta|\le N}\|\der w^\alpha\|_{\alpha+a,1}^2\notag \\
& \ \ \ \  + \norm \sum_{A=0}^a\sum_{|\beta'|\le {N-a}}\Big(
 \int_B  \psi w^{1+\alpha+a}\big|X_r^{A}\sn^{\beta'} D\uptheta\big|^2 + \int_B  \psi w^{1+\alpha+a}\big|X_r^{A}\sn^{\beta'} D\mathscr G\big|^2 \,dy\Big)\, \notag \\
 & \lesssim (1+\norm )\left(\norm +\sum_{a+|\beta|\le N}\|\der w^\alpha\|_{\alpha+a,1}^2 \right) \notag \\
 & \ \ \ \  + \norm\sum_{A=0}^a\sum_{|\beta'|\le {N-a}}\int_B  \psi w^{1+\alpha+a}\big|X_r^{A}\sn^{\beta'} D\mathscr G\big|^2\, dy, \label{E:PARTIALYGESTIMATE1}
\end{align}
where we recall~\eqref{E:BARW}.
The term in the last line above appears due to the quadratic terms that scale like $D\uptheta D\mathscr G$ on the right-hand sides of~\eqref{E:DIV3} and~\eqref{E:CURL3}. The presence of $ \sum_{a+|\beta|\le N}\|\der w^\alpha\|_{\alpha+a,1}^2$ is caused by the term $w^\alpha$ on the right-hand side of~\eqref{E:DIV3}.
The estimates follow from the standard Hardy-Sobolev embeddings and the definition of the norm $\norm$. Plugging~\eqref{E:PARTIALYGESTIMATE1} into~\eqref{E:PARTIALYGESTIMATE}, using the a priori bounds~\eqref{E:APRIORI1}--\eqref{E:APRIORI1.1} (to infer that $\norm$ is sufficiently small), from~\eqref{E:PARTIALYGESTIMATE0} and the inductive assumption~\eqref{E:INDUCTIVEASS} we conclude that 
\begin{align}
 \int_B  \psi w^{1+\alpha+a}\big|X_r^{a+1}\sn^\beta \mathscr G\big|^2 \,dy \lesssim \norm +\sum_{a+|\beta|\le N}\|\der w^\alpha\|_{\alpha+a,1}^2.
\end{align}
Estimate~\eqref{E:INTERIOREST} follows by a similar argument, wherein the norms of the derivatives of $w$ over the support of $(1-\psi)$ are all uniformly controlled by some constant. 
\end{proof}


Finally, a simple corollary of the previous lemma is the following estimate, stated in the form that will be used in the proof of the main theorem in Section~\ref{S:MAINTHEOREM}.

\begin{proposition}[Field error terms]\label{P:FIELD}
Let $\uptheta$ be a solution of~\eqref{E:THETAEQUATION}--\eqref{E:POTENTIALEQUATION} defined on a time interval $[0,T]$. Then for any $1<\gamma<\frac53$ there exists a constant $\mu_2>0$ such that for any 
$\tau\in[0,T]$ we have the bound
\begin{align}
& \delta^{\alpha-1}\sum_{a+|\beta|\le N} \int_0^\tau\mu(\sigma)^{3\gamma-4}\int_B  \psi w^{\alpha+a}\der\left(\A^k_i \Psi,_k\right) \der{\bf V}^i\,dy\,d\sigma \notag \\
& + \delta^{\alpha-1}\sum_{|\nu|\le N} \int_0^\tau\mu(\sigma)^{3\gamma-4}\int_B  (1-\psi)w^\alpha \pa^\nu\left(\A^k_i \Psi,_k\right) \pa^\nu{\bf V}^i\,dy\,d\sigma \notag \\
& \lesssim \int_0^\tau e^{-\mu_2\sigma}\mathcal S^N(\sigma) \,d\tau  + \delta^{\alpha-\frac12}\sum_{a+|\beta|\le N}\|\der w^\alpha\|_{\alpha+a,1} \int_0^\tau e^{-\mu_2\sigma}\sqrt{\mathcal S^N(\sigma)}\,d\sigma, \label{E:FIELDBOUND}
\end{align}
where we recall that $\tau\mapsto \mu(\tau)$ and $\mu_2>0$ are defined in part (a) of Lemma~\ref{L:ADELTAUNIFORM} and~\eqref{c2} respectively.
\end{proposition}


\begin{proof}
By the definition~\eqref{E:SNORM} of the norm $\norm$ we have the bound $\|\der{\bf V}^i\|^2_{a+\alpha,\psi}+\|\pa^\nu{\bf V}^i\|_{\alpha,1-\psi}^2 \le \delta \mu^{3-3\gamma} \norm$. We can therefore estimate the left-hand side of~\eqref{E:FIELDBOUND} by
\begin{align*}
&\delta^{\alpha-1} \sum_{a+|\beta|\le N} \int_0^\tau\mu(\sigma)^{\frac{3\gamma-5}{2}}\|\der\mathscr G\|_{\alpha+a,\psi}
\mu(\sigma)^{\frac{3\gamma-3}{2}}\|\der{\bf V}\|_{\alpha+a,\psi}\,d\sigma \notag \\
& +\delta^{\alpha-1} \sum_{|\nu|\le N} \int_0^\tau\mu(\sigma)^{\frac{3\gamma-5}{2}}\|\pa^\nu\mathscr G\|_{\alpha,1-\psi}
\mu(\sigma)^{\frac{3\gamma-3}{2}}\|\pa^\nu{\bf V}\|_{\alpha,1-\psi}\,d\sigma\notag  \\
& \lesssim  \delta^{\alpha-\frac12}\int_0^\tau \mu^{\frac{3\gamma-5}{2}} ( \sqrt{\norm(\sigma)} +\sum_{a+|\beta|\le N}\|\der w^\alpha\|_{\alpha+a,1}) \sqrt{\norm(\sigma)}\,d\sigma
\end{align*}
where we have used 
Lemma~\ref{L:KEY} in the last line.
Choosing $\mu_2$ as in \eqref{c2},  $\mu^{\frac{3\gamma-5}{2}}\lesssim e^{-\mu_2\sigma}$ and thus the claim follows.
\end{proof}


\begin{remark}
We note that the constant $\mu_2$ depends on $\gamma$ and $\lim_{\gamma\to\frac53}\mu_2(\gamma) =0$.
\end{remark}

\subsection{Vorticity estimates}

One of the key difficulties in controlling the dynamics of the EP flow are the vorticity bounds. It was recognised in~\cite{CLS, CoSh2012, JaMa2015} that the fact that the vorticity tensor satisfies a transport equation, which is partially decoupled from the divergence-part of the velocity, is sufficient to obtain ``good" estimates on the vorticity. In our case the quantity which satisfies a favourable transport equation is given by the modified vorticity $\text{Curl}_{\Lambda\A}{\bf V}$, where we recall that
${\bf V}=\pa_\tau\uptheta$. We first observe that equation~\eqref{E:THETAEQUATION} can be written in the form
\begin{align}\label{E:EPCURL0}
 \delta^{-1}\mu^{3\gamma-3}\left({\bf V}_\tau + \frac{\mu_\tau}{\mu}{\bf V}+2\Gamma^\ast {\bf V}\right) 
+ \Lambda\nabla_\eta\left(\frac\gamma{\gamma-1}(f^{\gamma-1}) + \delta^{\alpha-1}\mu^{3\gamma-4}\Psi\right) 
+\Lambda{\bf \eta} =0.
\end{align}
From the fact that $\Lambda\nabla_\eta$ and $\Lambda\eta$ are both annihilated by $\text{Curl}_{\Lambda\A}$ we conclude that
\be\label{E:EPCURLBASIC}
\text{Curl}_{\Lambda\A}{\bf V}_\tau + \frac{\mu_\tau}{\mu}\text{Curl}_{\Lambda\A}{\bf V}+2\text{Curl}_{\Lambda\A}\left(\Gamma^\ast{\bf V}\right)=0.
\ee
Equation~\eqref{E:EPCURLBASIC} is identical to equation (3.90) in~\cite{HaJa2} and therefore by an argument identical to the one presented in Section 3 of~\cite{HaJa2} we obtain the following proposition:
\begin{proposition}[\cite{HaJa2}, Prop. 3.3]\label{P:CURLESTIMATES}
Let $\gamma\in(1,\frac53)$ and $(\uptheta,{\bf V}): B\to\mathbb R^3\times \mathbb R^3$ be a unique solution to~\eqref{E:THETAEQUATION}--\eqref{E:THETAINITIAL} defined on some time interval $[0,T]$
and satisfying the a priori assumptions~\eqref{E:APRIORI1}--\eqref{E:APRIORI1.1}. 
Then for any $\tau\in [0,T]$ 
\begin{align}
&\vortnorm[{\bf V}](\tau) \lesssim  e^{-2\mu_0\tau}\left(\sqrt\delta\norm(0)+\vortnorm(0)\right)+ \sqrt\delta e^{-2\mu_0\tau}\norm(\tau)  
\label{E:CURLVBOUND} \\
&
\vortnorm[\uptheta](\tau) \lesssim  \norm(0)+\vortnorm(0) + \kappa\norm(\tau) + \int_0^\tau e^{-\mu_0\tau'} \norm(\tau')\,d\tau' \label{E:CURLTHETABOUND}
\end{align}
where 
$0<\kappa\ll 1$ is a small constant, $\mu_0$ is defined in~\eqref{E:MUDELTA0DEF}, and $\vortnorm(\bf V)$, $\vortnorm[\uptheta]$ are defined in~\eqref{E:VORTNORMDEF}.
\end{proposition}

We note in passing that the only difference to the estimates stated in Prop. 3.3 in~\cite{HaJa2} is the explicit occurrence of the factor $\sqrt \delta$
 in~\eqref{E:CURLVBOUND}. This factor is obtained by an explicit inspection of the proof from~\cite{HaJa2} and the fact that for any $\tau\in[0,T]$ we have the bound
 $\mu^{3\gamma-3}\sum_{a+|\beta|\le N}
\Big\{\left\|\der{\bf V}\right\|_{a+\alpha,\psi}^2+\sum_{|\nu|\le N}\left\|\pa^\nu{\bf V}\right\|_{\alpha,1-\psi}^2\Big\} \le \delta \norm$.

\subsection{Proof of the main theorem}\label{S:MAINTHEOREM}

Theorem~\ref{T:LOCAL} guarantees the existence of a time $T_\epsilon>0$ such that there exists a unique solution $(\uptheta(\tau,\cdot),{\bf V}(\tau,\cdot))$ on the time interval $[0,T]$ such that the map $[0,T]\ni \tau\mapsto \norm(\tau)$ is continuous and 
\[
\norm(\tau) \le 2 \varepsilon, \ \ \tau\in[0,T].
\]
A priori, the time of existence $T_\epsilon$ may converge to $0$ as $\epsilon$ goes to zero.
Let $\mathcal T$ be the maximal time of existence on which the map $\tau\to\norm(\tau)+\vortnorm(\tau) $ is continuous and
\be\label{E:MATHCALTDEF}
\norm(\tau)+\vortnorm(\tau) < 2C^*\left(\epsilon+\delta^{2\alpha-1}\right)
\ee
for a constant $C^*$ to be specified below.
Clearly $\mathcal T\ge T_\epsilon$.
Our first step is to show that on the time interval $[0,\mathcal T]$ the following energy bound holds
\begin{align}
\norm(\tau) \lesssim &\norm(0)+\vortnorm(0) 
+ \int_0^\tau e^{-\mu_\ast\sigma} \norm(\sigma)\,d\sigma \notag \\
& \ +\delta^{\alpha-\frac12}\sum_{a+|\beta|\le N}\|\der w^\alpha\|_{\alpha+a,\psi}\int_0^\tau e^{-\mu_2\sigma} \sqrt{\norm(\sigma)}\,d\sigma \label{E:ENERGYMAIN}
\end{align}
where $\mu_\ast=\min\{\mu_0,\mu_2\}$. 

To show~\eqref{E:ENERGYMAIN} we shall adopt a strategy developed in~\cite{HaJa2} which relies on the use of specially weighted multipliers, adapted to the 
the presence of the degenerate weights $w$ and the twisted gradient $\Lambda\nabla$ in~\eqref{E:THETAEQUATION}. Just like in~\cite{HaJa2} for any pair $(a,\beta)$ satisfying $a+|\beta|\le N$ we commute~\eqref{E:THETAEQUATION} with $\der$ and evaluate the $L^2$-inner product of the resulting equation with $\psi\Lambda_{im}^{-1}\der {\bf V}^m$ (the role of multiplication by $\psi$ is to localise the estimates to the neighbourhood of the boundary $\pa B_1(0)$.) Similarly, for any
multi-index $\nu$ satisfying $|\nu|\le N$ we commute~\eqref{E:THETAEQUATION} with $\pa^\nu {\bf V}$ and evaluate the $L^2$-inner product of the resulting equation with $(1-\psi)\Lambda_{im}^{-1}\pa^\nu{\bf V}^m$. This procedure leads to the following high-order estimate
\begin{align}
& \mathcal S^N(\tau) +\int_0^\tau \mathcal D^N(\tau')\,d\tau' \lesssim  \norm(0)  +\vortnorm[\uptheta](\tau) + \sqrt \delta \int_0^\tau  (\norm(\tau'))^{\frac12} (\vortnorm[{\bf V}](\tau'))^{\frac12}\,d\tau'  \notag \\
& \ \ \ \  + \kappa\norm(\tau) + \int_0^\tau e^{-\mu_0\tau'} \norm(\tau') d\tau' \notag \\
& \ \ \ \  + \delta^{\alpha-1}\sum_{a+|\beta|\le N} \int_0^\tau\mu(\sigma)^{3\gamma-4}\int_B  \psi w^{\alpha+a}\der\left(\A^k_i \Psi,_k\right) \der{\bf V}^i\,dy\,d\sigma  \notag \\
& \ \ \ \  +  \delta^{\alpha-1}\sum_{|\nu|\le N} \int_0^\tau\mu(\sigma)^{3\gamma-4}\int_B  (1-\psi)w^\alpha \pa^\nu\left(\A^k_i \Psi,_k\right) \pa^\nu{\bf V}^i\,dy\,d\sigma  \label{E:EPENERGYMAIN1}
\end{align}
where $0<\kappa\ll 1$ is a small constant, the constant $\mu_0$ is defined in \eqref{E:MUDELTA0DEF}, and the dissipation $\mathcal D^N$ is
defined as follows:
\begin{align*}
\mathcal D^N ({\bf V}) = \mathcal D^N(\tau) : =\frac{5-3\gamma}{\delta}\mu^{3\gamma-3} &\int_B   \Big[\psi \sum_{a+|\beta|\le N}  w^{a+\alpha} |\der{\bf V}|^2+ (1-\psi)\sum_{|\nu|\le N}  w^{\alpha} |\pa^\nu{\bf V}|^2 \Big]\,dy
\end{align*}
A detailed proof of how the terms in the first two lines on the right-hand side of~\eqref{E:EPENERGYMAIN1} is involved, but a precise derivation can be found in Section 4 of~\cite{HaJa2}. We note here the only noteworthy difference: in our case the weighted $L^2$-norms of the derivatives of the Lagrangian velocity ${\bf V}$ are additionally weighted by a factor of $\delta^{-1}$. This has to be taken into account when estimating the, generally speaking, cubic error terms. However, in each such cubic error integrand, the lowest order terms are always given in terms of purely spatial derivatives and require no $\delta$ weights in the estimates. This is allows us to prove the energy bound~\eqref{E:EPENERGYMAIN1}.

Observe that the assumption $\gamma<\frac53$ guarantees that $\mathcal D^N$ is positive. We may now use Proposition~\ref{P:CURLESTIMATES} to bound the vorticity norm $\mathcal B^N$ in terms of the norm $\mathcal S^N$ and Proposition~\ref{P:FIELD} to control the last two lines on the right-hand side of~\eqref{E:EPENERGYMAIN1}. Using the smallness of $\kappa$ and a priori assumptions~\eqref{E:APRIORI1}--\eqref{E:APRIORI1.1} 
this leads to the estimate
\begin{align}
\norm(\tau) + \int_0^\tau \mathcal D^N(\tau')\,d\tau' \lesssim & \, \norm(0)+\vortnorm(0) + \int_0^\tau e^{-\mu_\ast\tau'}\mathcal S^N(\tau') \,d\tau'  \notag \\
&  + \delta^{\alpha-\frac12}\sum_{a+|\beta|\le N}\|\der w^\alpha\|_{\alpha+a,1} \int_0^\tau e^{-\mu_\ast\tau'}\sqrt{\norm(\tau')}\,d\tau', \label{E:C1}
\end{align}
where $\mu_\ast = \min\{\mu_0,\mu_2\}$.
Using the Young inequality we can estimate $\delta^{\alpha-\frac12}\sum_{a+|\beta|\le N}\|\der W^\alpha\|_{\alpha,a}\int_0^\tau e^{-\mu_\ast\sigma}\sqrt{\norm(\sigma)}\,d\sigma\le  \int_0^\tau e^{-\mu_\ast\tau}\norm(\tau')\,d\tau' 
+ C  \delta^{2\alpha-1}$ where we have used the positivity of $\mu_\ast$, the bound~\eqref{E:WBOUND}, and the uniformity of $\mu_0=\frac{3\gamma-3}{2}\mu_1,\mu_2=\frac{5-3\gamma}{2}\mu_1$ with respect to $\delta$ as expressed in Lemma~\ref{L:ADELTAUNIFORM} and~\eqref{E:MU1DELTADEF}--\eqref{c2}. 
Combining this with~\eqref{E:C1} we obtain
\begin{align}
\norm(\tau) + \int_0^\tau \mathcal D^N(\tau')\,d\tau' \le & \, \bar C\left(\norm(0)+\vortnorm(0) + \int_0^\tau e^{-\mu_\ast\tau'}\norm(\tau')\,d\tau' +\delta^{2\alpha-1}\right), \label{E:C2}
\end{align}
for some constant $\bar C>0$. Given a small number $0<\varepsilon'\ll1$, 
by a classical well-posedness estimate we may conclude that there exists an $\varepsilon>0$ and $C>0$ such that if $\norm(0)+\vortnorm(0)+\delta^{2\alpha-1}<\varepsilon$, then the solution exists on a sufficiently long time interval $[0,T^*]$ satisfying $T^*\ge \frac1{\mu_\ast}|\log(\varepsilon'\mu_\ast)|$
and
\begin{align}
\norm(\tau) + \int_0^\tau \mathcal D^N(\tau')\,d\tau' +\vortnorm(\tau)\le C\left( \norm(0)+\vortnorm(0)+\delta^{2\alpha-1}\right)
\end{align}

For any $\tau\in(T^*,\mathcal T)$ we have from~\eqref{E:C2} the bound,
\begin{align*}
\norm(\tau) + \int_{T^*}^\tau \mathcal D^N(\tau')\,d\tau' & \le  \,\bar C\left( \norm(T^*)+\vortnorm(T^*) + \int_{T^*}^\tau e^{-\mu_\ast\tau'}\norm(\tau')\,d\tau' +\delta^{2\alpha-1}\right) \\
& \le \bar C\left( \norm(T^*)+\vortnorm(T^*)+\delta^{2\alpha-1} \right)+ \bar C\frac1{\mu_\ast}e^{-\mu_\ast T^*}\norm(\tau) \\
& \le  C\bar C\left( \norm(0)+\vortnorm(0)+\delta^{2\alpha-1}\right) + \epsilon'\norm(\tau),
\end{align*}
where we have used the bound on $T^*$ in the last line above. With $\epsilon'$ sufficiently small we conclude that 
\begin{align*}
\norm(\tau) + \int_{T^*}^\tau \mathcal D^N(\tau')\,d\tau' < 2C^*\left( \norm(0)+\vortnorm(0)+\delta^{2\alpha-1}\right),
\end{align*}
where $C^*= C \bar C$.
Therefore $\mathcal T=+\infty$ by definition of $\mathcal T$. In addition, it is easy to see that the a priori bounds in \eqref{E:APRIORI1} and \eqref{E:APRIORI1.1} are improved. This can be checked by the use of the fundamental theorem of calculus in $\tau$-variable and the above energy bound.

\section*{Acknowledgements}

The authors express their gratitude to Yan Guo for helpful discussions.
JJ is supported in part by NSF DMS-1608494. 
MH acknowledges the support of the EPSRC Grant EP/N016777/1.
\appendix

\section{Local-in-time well-posedness}\label{S:LWP}

\setcounter{equation}{0}

Assume that $B_0$ is a simply connected domain  diffeomoprhic to the unit ball $B$ in $\mathbb R^3$. In other words, there exists a diffeomoprhism $\chi:B\to B_0$
which takes the unit sphere $\pa B$ to the boundary $\pa B_0$ of $B_0$. In particular we allow our initial geometry $B_0$ to be non-convex and not a perturbation of the unit ball. 
Let $t\to\zeta(t,\cdot)$ be the Lagrangian flow map associated with the Euler-Poisson system~\eqref{E:EULERPOISSON}. In other words, $\zeta$ solves
\begin{align}
\zeta_t(t,x) & = {\bf u}(t,\zeta(t,x)), \ x\in B, \\
\zeta(0,x) & = \chi(x), \ x\in B.
\end{align}

To formulate the EP$_\gamma$ system on a fixed domain we need to introduce some additional notation. In analogy to Section~\ref{SS:LAGRCOORD} we introduce
\begin{align}
A : = [D\zeta]^{-1} \ \ &\text{ (Inverse of the Jacobian matrix)},\\
J  : = \det [D\zeta] \ \ &\text{ (Jacobian determinant)},\\
v : = {\bf u}\circ \zeta \ \ &\text{ (Lagrangian modified velocity)},\\
\Psi : = \Phi\circ\zeta \ \ &\text{ (Lagrangian potential)}.
\end{align}

It is then straightforward to check that the EP$_\gamma$ system takes the form
\begin{align}
\zeta^i_{tt} + \left(w^{1+\alpha} A^k_i J^{-\frac1\alpha}\right),_k + A^k_i\Psi,_k & = 0 \\
A^k_i(A^j_i\Psi,_j),_k = 4\pi w^\alpha J^{-1}, \label{E:LOCLAGR}
\end{align}
where $w:=\rho_0^{\gamma-1}\circ\chi J_0^{\gamma-1}$, $J_0:=\det [D\chi]$, and $\alpha=\frac1{\gamma-1}$.  We treat the problem~\eqref{E:LOCLAGR} as an initial value problem, thus assuming further that 
\begin{align}\label{E:LOCLAGRINITIAL}
\zeta(0,\cdot)= \chi, \ \ v(0,\cdot) = v_0.
\end{align}
Let $\sqrt g$ denote the volume element associated with the surface geometry of $\pa B_0$. Then the following formula holds
\begin{align}
n_i = (A_0)^k_i N_k \frac{J_0}{\sqrt g}, \ \ i=1,2,3.
\end{align}
Using this formula it is easy to see that 
\[
\partial_n (\rho_0^{\gamma-1})\circ\chi = (A_0)^k_i\pa_kw J_0^{1-\gamma} n_i = (A_0)^k_i N_k\pa_Nw J_0^{1-\gamma} n_i = \sqrt g J_0^{-\gamma} \pa_N w.
\]
In particular, the physical vacuum condition~\eqref{E:PHYSICALVACUUM} implies that
\be\label{E:GOODSIGN}
\pa_N w<0.
\ee
Here $\pa_N$ denotes the (outward) normal derivative with respect to the unit ball $ B$. 

To prove well-posedness we shall  use the following norm:
\begin{align}
& \locnorm(\zeta,v)(t)  = \locnorm(t)  \notag \\ 
&: =   \sum_{a+|\beta|\le N}\sup_{0\le\tau\le t}
\Big\{\left\|\der v\right\|_{a+\alpha,\psi}^2  + \left\|\nabla_\zeta\der \zeta\right\|_{a+\alpha+1,\psi}^2
 +\left\|\text{div}_\zeta\der \zeta\right\|_{a+\alpha+1,\psi}^2 \Big\} \notag \\
& \ \ \ \ +\sum_{|\nu|\le N}\sup_{0\le\tau'\le\tau}
\Big\{\left\|\pa^\nu v\right\|_{\alpha,1-\psi}^2+\left\|\nabla_\zeta\pa^\nu \zeta\right\|_{\alpha+1,1-\psi}^2
+\left\|\text{div}_\zeta\pa^\nu \zeta\right\|_{\alpha+1,1-\psi}^2\Big\}  \label{E:LOCSNORM} 
\end{align}
Additionally we introduce another high-order quantity measuring the modified vorticity of $ v$ which is a priori not controlled by the norm $\mathcal S^N(\tau)$: 
\begin{align}
\locvortnorm(v)(t)=\locvortnorm(t)  & :=\sum_{a+|\beta|\le N} \sup_{0\le\tau \le t}\left\|\text{Curl}_{\zeta}\der v\right\|_{a+\alpha+1,\psi}^2  +\sum_{|\nu|\le N} \sup_{0\le\tau\le t} \left\|\text{Curl}_{\zeta}\pa^\nu v\right\|_{\alpha+1,1-\psi}^2.\label{E:LOCVORTNORMDEF}
\end{align}
The weighted spaces $\|\cdot\|_{a+\alpha,\psi}$ etc., have been introduced in~\eqref{E:WEIGHTEDSPACES} and similarly. The tangential operators $\sn^\beta$ have been introduced in Section~\ref{SS:NOTATION}

\begin{theorem}\label{T:LOCAL}
Let $\gamma>1$ and assume that the physical vacuum condition~\eqref{E:PHYSICALVACUUM} is satsfied. If $N\in\mathbb N$ satisfies $N\ge \frac2{\gamma-1}+12$ then for any initial data
$(\zeta_0,v_0)$ satisfying $\locnorm(\zeta_0,v_0)+\locvortnorm(v_0)<\infty$, there exists a time $T>0$ and a unique solution $t\to(\zeta(t,\cdot))$ of the initial value problem~\eqref{E:LOCLAGR}--\eqref{E:LOCLAGRINITIAL} such that the map $[0,T]\ni t\mapsto \locnorm(t)+\locvortnorm(t)\in\mathbb R$ is continuous and the solution satisfies the bound
\[
\locnorm(t)+\locvortnorm(t) \le  C_0,
\]
where the constant $C_0$ depends only on the initial conditions.
\end{theorem}

\noindent
{\em Sketch of the proof of Theorem~\ref{T:LOCAL}.}
The proof of Theorem~\ref{T:LOCAL} follows by incorporating two ingredients: the well-posedness proof for the compressible Euler system of Jang \& Masmoudi~\cite{JaMa2015} and the estimates on the gravitational potential in Section~\ref{SS:FORCEFIELD}. The basic idea is to show an energy estimate of the form
\[
\locnorm(t) + \locvortnorm(t) \le C + t p(\locnorm(t)) +tp( \locvortnorm(t)), 
\]
where $p(s) = P(\sqrt s)$ and $P$ is a polynomial of degree at least 1.
The novelty with respect to~\cite{JaMa2015} are the estimates of the potential term that in the analysis present themselves as error terms of the form
\begin{align}
& \sum_{a+|\beta|\le N} \int_0^t\int_B  \psi w^{a}\der\left(w^a\left(\A^k_i \Psi,_k\right)\right) \der v^i\,dy\,d\tau \notag \\
& \ \ \ \ +  \sum_{|\nu|\le N} \int_0^t \int_B  (1-\psi)\pa^\nu\left(w^a\left(\A^k_i \Psi,_k\right)\right) \pa^\nu v^i\,dy\,d\tau. \label{E:FIELDBOUNDLOC}
\end{align}
By a similar argument as in the proof of Proposition~\ref{P:FIELD}, this time estimating the integrands inside the $\int_0^t\dots \,d\tau$ integrals in $L^\infty([0,t])$-norm, 
we show that the expression~\eqref{E:FIELDBOUNDLOC} is bounded by  a constant multiple of 
\[
t \locnorm(t) + C t,
\]
where the constant $C$ depends only on the initial conditions. A classical continuity argument yields the desired a priori bounds. With the a priori bounds, the same approximate scheme at $(n+1)^{th}$ step used in \cite{JaMa2015} with the potential term labeled by $n$ (lower order term) together with the duality argument leads to the local well-posedness.


\begin{thebibliography}{99}



\bibitem{BiTr}
J. Binney, S. Tremaine.
\textit{Galactic Dynamics}.
Princeton University Press, Princeton, 2008.

%
\bibitem{Ch} 
Chandrasekhar, S. 
\textit{An Introduction to the Study of Stellar Structures.} 
University of Chicago Press, Chicago, 1938.


\bibitem{CoSh2011} Coutand, D., Shkoller, S.
Well-posedness in smooth function spaces for the moving-boundary 1-D compressible Euler equations in physical vacuum. 
{\em Comm. Pure Appl. Math.} {\bf 64} (2011), no. 3, 328--366. 

\bibitem{CoSh2012}
Coutand, D., Shkoller, S.
Well-posedness in smooth function spaces for the moving boundary three-dimensional compressible Euler equations in physical vacuum.
{\em Arch. Ration. Mech. Anal.} {\bf 206} (2012), no. 2, 515--616.

\bibitem{CLS} 
Coutand, D., Lindblad, H., Shkoller, S.
A priori estimates for the free--boundary 3D compressible Euler equations in physical vacuum.
{\em Comm. Math. Phys.} {\bf 296} (2010), 559--587.


\bibitem{DM}
Dacorogna B., Moser, J.
On a partial differential equation involving the Jacobian determinant.
{\em Ann. Inst. H. Poincar\'{e} Anal. Non Lin\'{e}aire} \textbf{7} (1990), no. 1, 1--26. 


\bibitem{DengXiangYang}
Deng, Y., Xiang, J., Yang, T.
Blowup phenomena of solutions to Euler-Poisson equations. 
\emph{J. Math. Anal. Appl.} \textbf{286} (2003), 295--306.  



\bibitem{GMP2013}
 Germain, P.,  Masmoudi, N., Pausader, B.
Non-neutral global solutions for the electron Euler-Poisson system in 3D.
{\em Siam. J. Math. Anal.}
{\bf 45-1} (2013), 267--278.

\bibitem{GoWe}
Goldreich, P.,  Weber,  S. 
Homologously collapsing stellar cores. 
\emph{Astrophys. J.} \textbf{238} (1980),  991--997. 

\bibitem{Grassin98}
Grassin, M.
\newblock Global smooth solutions to Euler equations for a perfect gas.
\newblock {\em Indiana Univ. Math. J.} \textbf{47} (1998), 1397-1432.

\bibitem{GuLe}
Gu, X., Lei, Z.
Local Well-posedness of the three dimensional compressible Euler--Poisson equations with physical vacuum.
{\em Journal de Mathématiques Pures et Appliquées}
{\bf 105}, 5 (2016), 662--723.

\bibitem{Guo1998}
 Guo, Y.
Smooth irrotational flows in the large to the Euler-Poisson system in $\mathbb R^{3+1}$. 
{\em Comm. Math. Phys.} 
{\bf 195},
no. 2  (1998), 249--265.


\bibitem{GuoIonescuPausader2016}
Guo, Y., Ionescu, A. D., Pausader, B.
Global solutions of the Euler-Maxwell two-fluid system in 3D. 
{\em Ann. of Math. (2)} 
{\bf 183}, no. 2  (2016), 377--498. 


\bibitem{GuoPausader2011}
 Guo, Y., Pausader, B.  
Global smooth ion dynamics in the Euler-Poisson system. 
{\em Comm. Math. Phys.} 
{\bf 303},
no. 1 (2011), 89--125.

\bibitem{GuoZadeh1998}
 Guo, Y., Tahvildar-Zadeh, A. S.
Formation of singularities in relativistic fluid dynamics and in spherically symmetric plasma dynamics. 
{\em Contemp. Math.}, 
{\bf 238}, 
Amer. Math. Soc., Providence, RI (1999)


\bibitem{HaJa}
Had\v zi\'c, M.,  Jang, J.
Nonlinear stability of expanding star solutions in the radially-symmetric mass-critical Euler-Poisson system.
{\em Comm. Pure Appl. Math.}, DOI: 10.1002/cpa.21721

\bibitem{HaJa2}
Had\v zi\'c, M.,  Jang, J. 
Expanding large global solutions of the equations of compressible fluid mechanics. Preprint, arXiv:1610.01666

\bibitem{HaJa3}
Had\v zi\'c, M.,  Jang, J.
Dynamics of expanding gases.  
To appear in
{\em Research Institute for Mathematical Science, Kyoto, Kôkyûroku, No. 2038, Mathematical Analysis in Fluid and Gas Dynamics.}

\bibitem{HaJa4}
Had\v zi\'c, M.,  Jang, J.
Nonlinear stability of expanding star solutions in the mass-critical Euler-Poisson system.
{\em In preparation.}


\bibitem{IonescuPausader2013}
  Ionescu, A., Pausader, B.
The Euler-Poisson system in 2D:  global stability of the constant equilibrium solution. 
{\em Int. Math. Res. Notices} (4)  (2013), 761--826.


\bibitem{J0} 
Jang, J. 
Nonlinear Instability in Gravitational Euler-Poisson system for $\gamma=6/5$.
\textit{ Arch. Ration. Mech. Anal.} \textbf{188} (2008), 
265--307.

\bibitem{Jang2012} 
 Jang, J.
The two-dimensional Euler-Poisson system with spherical symmetry. 
{\em J. Math. Phys.}
{\bf 53},  (2012).

\bibitem{JangLiZhang2014} 
  Jang, J., Li, D., Zhang, X.
Smooth global solutions for the two-dimensional Euler-Poisson system. 
{\em Forum Math.} 
{\bf 26} (2014), 645--701. 


\bibitem{Jang2014}
Jang, J.
Nonlinear Instability Theory of Lane-Emden stars.  
{\em Comm. Pure Appl. Math.} {\bf 67} (2014), no. 9, 1418--1465.

\bibitem{JaMa2009} 
Jang, J., Masmoudi, N.
Well-posedness for compressible Euler equations with physical vacuum singularity.
{\em Comm. Pure Appl. Math.} {\bf 62} (2009), 1327--1385.

\bibitem{JM1} 
Jang, J., Masmoudi, N.
Vacuum in Gas and Fluid dynamics. 
\textit{Proceedings of the IMA summer school on Nonlinear Conservation Laws and Applications}, Springer (2011), 315--329.

\bibitem{JM2012}
Jang, J., Masmoudi, N. 
Well and ill-posedness for compressible Euler equations with vacuum. 
\emph{J. Math. Phys.} \textbf{53} (2012), 115625. 


\bibitem{JaMa2015}
Jang, J., Masmoudi, N.
Well-posedness of compressible Euler equations in a physical vacuum.
{\em Communications on Pure and Applied Mathematics} {\bf 68} (2015), no. 1, 61--111.
%

\bibitem{LiWu2014}
  Li, D., Wu, Y.
The Cauchy problem for the two dimensional Euler-Poisson system. 
{\em J. Eur. Math. Soc.} 
{\bf 10}  (2014), 2211--2266.


\bibitem{L2} 
Liu, T.-P.
Compressible flow with damping and vacuum. 
\textit{Japan J. Appl.
Math} \textbf{13} (1996), 25-32.

\bibitem{LS}
Liu, T.-P., Smoller, J.
On the vacuum state for isentropic gas dynamics equations.
\textit{Advances in Math.} \textbf{1}  (1980), 345-359.

\bibitem{LY1} 
Liu, T.-P., Yang, T.
Compressible Euler equations with vacuum. 
\textit{J. Differential
Equations}  \textbf{140} (1997), 223-237.

\bibitem{LY2} 
Liu, T.-P., Yang, T.
Compressible flow with vacuum and physical singularity.
\textit{Methods Appl. Anal.} \textbf{7} (2000), 495-509.




\bibitem{LuSm}
 Luo, T., Smoller, J.
Existence and Nonlinear Stability of Rotating Star Solutions of the Compressible Euler-Poisson Equations.
\emph{Arch. Ration. Mech. Anal.} {\bf 191} (2009), 3,  447--496.



\bibitem{LXZ} 
Luo, T., Xin, Z., Zeng, H.  
Well-posedness for the motion of physical vacuum of the three-dimensional compressible Euler equations with or without self-gravitation. 
\emph{Arch. Ration. Mech. Anal.} \textbf{213}, no. 3 (2014), 763-831.


\bibitem{Makino92}
Makino, T.
Blowing up solutions of the Euler-Poisson equation for the evolution of gaseous stars.
{\em Transport Theory Statist. Phys.} \textbf{21} (1992), 615--624.

\bibitem{MaPe1990}
Makino, T., Perthame, B.
Sur les Solution \'a  Sym\'etrie Sph\'erique de l'Equation d'Euler-Poisson
pour l'Evolution d'Etoiles Gazeuses.
{\em Japan J. Appl. Math.} {\bf 7} (1990), 165--170.


\bibitem{Rein} 
Rein, G.
Non-linear stability of gaseous stars. \emph{Arch. Ration. Mech. Anal.} \textbf{168} (2003), no. 2, 115--130.

\bibitem{Ro} 
Rozanova, O.
Solutions with linear profile of velocity to the Euler equations in several dimensions. 
{\em Hyperbolic problems: theory, numerics, applications, 861-870, Springer, Berlin,} 2003




\bibitem{Se1997}
Serre, D.
Solutions classiques globales des \'equations d'Euler pour un fluide parfait compressible. 
{\em Annales de l'Institut Fourier} 
{\bf 47} (1997), 139--153.

\bibitem{Se}
Serre, D.
Expansion of a compressible gas in vacuum. 
{\em Bulletin of the Institute of Mathematics, Academia Sinica, Taiwan.} 
{\bf 10} (2015),  695--716. 


\bibitem{ShSi2017}
Shkoller, S., Sideris, T. C.
Global existence of near-affine solutions to the compressible Euler equations.
{Preprint, arXiv:1710.08368}



\bibitem{Sideris2014}
Sideris, T. C. Spreading of the free boundary of an ideal fluid in a vacuum. 
{\em J. Differential Equations}, {\bf 257}(1) (2014), 1--14.

\bibitem{Sideris}
Sideris, T., C. Global existence and asymptotic behavior of affine motion of 3D ideal fluids surrounded by vacuum.  
{\em Arch. Ration. Mech. Anal.} \textbf{225}, no. 1 (2017), 141--176.
%

\bibitem{ZeNo}
Zel'dovich, Y. B.,  Novikov, I. D.
\textit{Relativistic Astrophysics Vol. 1: Stars and Relativity. }
Chicago University Press, Chicago, 1971.




\end{thebibliography}
\end{document}